\documentclass[11pt,oneside,reqno]{amsart}

\pdfoutput=1
\usepackage{amssymb}
\usepackage{algpseudocode}
\usepackage{algorithm}
\usepackage{verbatim}
\usepackage{graphicx}
\usepackage{placeins}
\usepackage{listings}
\usepackage{xcolor}
\usepackage{subcaption}
\usepackage[noadjust]{cite}
\RequirePackage{dsfont} 
 \scrollmode
\allowdisplaybreaks
\usepackage{graphicx}
\usepackage{epstopdf}
\usepackage{hyperref}
\usepackage{amsmath}
\usepackage{tikz}

\makeatletter
\newcommand{\xleftrightarrow}[2][]{\ext@arrow 3359\leftrightarrowfill@{#1}{#2}}
\makeatother

\definecolor{codegreen}{rgb}{0,0.6,0}
\definecolor{codegray}{rgb}{0.5,0.5,0.5}
\definecolor{codepurple}{rgb}{0.58,0,0.82}
\definecolor{backcolour}{rgb}{0.95,0.95,0.92}

\lstdefinestyle{list_style}{
  backgroundcolor=\color{backcolour}, commentstyle=\color{codegreen},
  keywordstyle=\color{magenta},
  numberstyle=\tiny\color{codegray},
  stringstyle=\color{codepurple},
  basicstyle=\ttfamily\footnotesize,
  breakatwhitespace=false,         
  breaklines=true,                 
  captionpos=b,                    
  keepspaces=true,                 
  numbers=left,                    
  numbersep=5pt,                  
  showspaces=false,                
  showstringspaces=false,
  showtabs=false,                  
  tabsize=2
}

\newcommand{\xdasharrow}[2][->]{
\tikz[baseline=-\the\dimexpr\fontdimen22\textfont2\relax]{
\node[anchor=south,font=\scriptsize, inner ysep=1.5pt,outer xsep=2.2pt](x){#2};
\draw[shorten <=3.4pt,shorten >=3.4pt,dashed,#1](x.south west)--(x.south east);
}
}

\newcommand{\DEBUG}{}

\ifdefined\DEBUG%

  \def\rem#1{{\marginpar{\raggedright\scriptsize #1}}}
  \newcommand{\pmr}[1]{\rem{\color{blue}{$\bullet$ #1}}}
  \newcommand{\ppr}[1]{\rem{\color{red}{$\bullet$ #1}}}
 \else%

  \newcommand{\ppr}[1]{}
  \newcommand{\pmr}[1]{}
 \fi

\newcommand{\R}{{\mathbb R}}

\newcommand{\N}{{\mathbb N}}

\newcommand{\one}{\cdot\mathds{1}}

\newcommand{\cost}{\operatorname{cost}}

\newcommand{\euler}{{\widetilde{X}_{M,n}^E}}
\newcommand{\xstep}{X_{M_n^*,k_n^*}^{step}}
\newcommand{\xalg}{{\overline{X}_{M_n,\bar{k}_n}}}

\newcommand{\cala}{{\mathcal A}}

\def\rd{\,{\mathrm d}}

\theoremstyle{plain}
\newtheorem{theorem}{Theorem}
\newtheorem{lemma}{Lemma}
\newtheorem{fact}{Fact}

\newtheorem{proposition}{Proposition}

\theoremstyle{definition}
\newtheorem{remark}{Remark}

\begin{document}

\title
[Global approximation of SDEs driven by countably dimensional Wiener process]
{Adaptive step-size control for global approximation of SDEs driven by countably dimensional Wiener process}

\author[{\L}. St\c{e}pień]{{\L}ukasz St\c{e}pie\'{n}}
\address{AGH University of Science and Technology,
	Faculty of Applied Mathematics,
	Al. A.~Mickiewicza 30, 30-059 Krak\'ow, Poland}
\email{lstepie@agh.edu.pl, corresponding author}

\begin{abstract}
In this paper we deal with global approximation of solutions of stochastic differential equations (SDEs) driven by countably dimensional Wiener process. Under certain regularity conditions imposed on the coefficients, we show lower bounds for exact asymptotic error behaviour. For that reason, we analyse separately two classes of admissible algorithms: based on equidistant, and possibly not equidistant meshes. Our results indicate that in both cases, decrease of any method error requires significant increase of the cost term, which is illustrated by the product of cost and error diverging to infinity. This is, however, not visible in the finite dimensional case. In addition, we propose an implementable, path-independent Euler algorithm with adaptive step-size control, which is asymptotically optimal among algorithms using specified truncation levels of the underlying Wiener process. Our theoretical findings are supported by numerical simulation in Python language.

\textbf{Key words:} stochastic differential equations, countably dimensional Wiener process, adaptive step-size control, minimal error bounds, Information--based Complexity, asymptotically optimal algorithm
\newline
\newline
\textbf{MSC 2010:} 65C30, 68Q25
\end{abstract}
\maketitle

\section{Introduction}
We investigate global approximation of solutions of the following stochastic differential equations 
\begin{equation}
\label{main_equation}
	\left\{ \begin{array}{ll}
	\displaystyle{
	\rd X(t) = a(t,X(t))\rd t + \sigma(t) \rd W(t) , \ t\in [0,T]},\\
	X(0)=x_0, 
	\end{array} \right.
\end{equation}
where $T >0$, $W(t) = [W_1(t), W_2(t), \ldots]^T$ is a sequence of independent scalar Wiener processes on the probability space with sufficiently rich filtration $(\Omega, \mathcal{F}, \mathbb{P}, \big(\mathcal{F}_t)_{t\in[0,T]}\big),$ and $x_0 \in \mathbb{R}.$ For suitable, regular coefficients $a, \sigma,$ the uniqueness of the solution $X=X(t),$ and its finite second-order moments can be assured; see \cite{applebaum, CohEl, Krylov1, PPMSLS} where more general models were considered.

Recently, global approximation of solutions of SDEs driven by finite-dimensional Wiener process has been studied extensively in the literature. In particular, the algorithms with step-size were introduced in \cite{Muller, Muller2, AKPP, PP1, PP2}. Generally, the time-step adaptation linked to the equation coefficients instead of leveraging equidistant mesh can significantly decrease the asymptotic constant for the method error in the finite dimensional models. On the other hand, SDEs driven by countably dimensional Wiener noise can be found in \cite{zabczyk, Cao, Liang06}, while their applications - e.g., in \cite{carmona, PBL}. Nonetheless, there are still few papers referring to the exact error behaviour and optimality issues for such SDEs in global setting. For instance, in \cite{SMST} authors developed an Euler algorithm and estimated its global error for $X$ being countably dimensional. However, the assumptions were relatively strong, and the proposed algorithm was non-implementable due to infinite dimension of the solution. 

In this paper we extend some asymptotic results for global approximation from \cite{Muller, AKPP} to SDEs with countably dimensional noise structure. To that end, we utilise solution moment bounds and approximation strategy presented in \cite{PPMSLS}, where a pointwise setting was investigated. In our setting, error of an algorithm $\cala$ returning a process $Y = (Y(t))_{t\in[0,T]}$ is measured in the norm $\|\cdot\|_2$ defined as follows
\begin{equation*}
    \|X-Y\|_2 = \left(\mathbb{E}\int_0^T |X(t)-Y(t)|^2 \rd t \right)^{1/2}.
\end{equation*}

Under suitable conditions imposed on the model coefficients, we analyse asymptotic exact error behaviour in two classes $\chi_{eq}, \chi_{noneq}$ of admissible algorithms leveraging only finite dimensional evaluations of the process~$W;$ we refer to \cite{Hein1, Muller, AK22, AKPP, NOV, PPVSMS2, TWW88} where similar approach for generic error analysis was developed. In our setting, permitted truncation levels are determined by non-decreasing sequences leveraging information about how fast $\sigma$ entries vanish. In Theorem \ref{th:AC_lower} we show that for any fixed, admissible truncation level sequence $\bar{M}=(M_n)_{n=1}^{\infty}$, an exact asymptotic behaviour of the cost-error relation satisfies
\begin{equation}\label{eq:AC_noneq}
    \big(\cost{(\xalg)}\big)^{1/2}\,\big\|\xalg - X\big\|_{2} \gtrapprox \frac{M_n^{1/2}}{\sqrt{6}}\int_{0}^T \|\sigma(t)\|_{\ell^2}\rd t, 
\end{equation}
irrespective of the choice of admissible method $(\xalg)_{n\in\mathbb{N}}\in \chi_{noneq}^{\bar{M}}$ based on (possibly) non-equidistant mesh with a suitable number of nodes $\bar{k_n}+1,$ and such that $\xalg$ utilises discrete information from $M_n, n\in\mathbb{N},$ first coordinates of $W.$
When we limit ourselves to methods $\chi_{eq}^{\bar{M}}$ based on equidistant partitions of the interval $[0,T],$ we get
\begin{equation}\label{eq:AC_eq}
    \big(\cost{(\xalg)}\big)^{1/2}\,\big\|\xalg - X\big\|_{2} \gtrapprox M_n^{1/2}\sqrt{\frac{T}{6}}\biggr(\int_{0}^T \|\sigma(t)\|_{\ell^2}^2\rd t \biggr)^{1/2}.
\end{equation}
We hereinafter use the notation
\begin{equation*}
\mathcal{C}_{noneq} = \frac{1}{\sqrt{6}}\int_{0}^T \|\sigma(t)\|_{\ell^2}\rd t, \quad\quad \mathcal{C}_{eq} = \sqrt{\frac{T}{6}}\biggr(\int_{0}^T \|\sigma(t)\|_{\ell^2}^2\rd t \biggr)^{1/2}.
\end{equation*}
By the H\"{o}lder inequality, we have $0 \leq \mathcal{C}_{noneq}\leq \mathcal{C}_{eq}$. We also stress that the lower bounds in \eqref{eq:AC_noneq} and \eqref{eq:AC_eq} diverge as $n$ tends to infinity, which illustrates significant increase of the informational cost needed to decrease the method error. Next, for fixed $\bar{M},$ we construct truncated Euler algorithm with adaptive path-independent step-size control $X^{step}_{M_n,k_n^*}$ which is optimal in class $\chi_{noneq}^{\bar{M}},$ since it attains asymptotic lower bound in \eqref{eq:AC_noneq}. We also provide lower bounds which hold irrespective of the truncation level sequence, see Theorem \ref{th:AC_final}. While those cannot be asymptotically achieved by any admissible algorithm, we show that the errors for the methods proposed in this paper can be arbitrary close in some sense to the obtained bound. We note that both error bounds and optimality are investigated in the spirit of IBC (Information-Based Complexity) framework.

According to our best knowledge, this is the first paper to establish lower bounds for exact asymptotic error in the global approximation setting for SDEs with countably dimensional Wiener process. Moreover, the new constructed algorithms are implementable, and their performance is verified by using the multiprocessing library in Python.

The paper is organised as follows. In Section \ref{sec:preliminaries} we provide basic notation, model assumptions and properties of the underlying solution $X.$ Then, in Section \ref{sec:lower_bound} we investigate lower bounds for exact asymptotic error behaviour. In Section \ref{sec:equidistant} and in Section \ref{sec:step_size} we introduce and show optimality of the truncated dimension Euler schemes in the classes $\chi_{eq}^{\bar{M}}$ and $\chi_{noneq}^{\bar{M}},$ respectively. Next, in Section \ref{sec:optimal_alg}, we extend optimality investigation to the classes $\chi_{eq}$ and $\chi_{noneq}.$ Finally, Section \ref{sec:experiments} deals with numerical experiments in Python and alternative solver implementation utilizing Numba compiler.
\section{Preliminaries}\label{sec:preliminaries}

Let $(\Omega,\Sigma,\mathbb{P})$ be a complete probability space and $\mathcal{N}_0=\{A\in\Sigma \ | \ \mathbb{P}(A)=0\}$. Let also $(\Sigma_t)_{t\geq  0}$ be a filtration on $(\Omega,\Sigma,\mathbb{P})$ that satisfies the usual conditions, i.e., $\mathcal{N}_0\subset\Sigma_0$ and is right-continuous.
Let $W=[W_1,W_2,\ldots]^T$ be a countably dimensional $(\Sigma_t)_{t\geq  0}$-Wiener process defined on $(\Omega,\Sigma,\mathbb{P})$. We note that similarly to the finite dimensional case, stochastic integrals with respect to $W$ enjoy properties such as the Burkholder inequality, It\^{o} isometry, It\^{o} lemma, see e.g. \cite{Cao, CohEl, PPMSLS}.

For $x\in \ell^2(\mathbb{R})$ we use the following notation $x = (x_1, x_2, \ldots).$ We introduce projection operators
$P_k:  \ell^2 (\mathbb{R}) \mapsto  \ell^2 (\mathbb{R}),$ $k\in\mathbb{N}\cup\{\infty\}$ with
\begin{equation*}
	P_k v = (v_1, v_2, \ldots, v_k, 0, 0,  \ldots), \quad   v\in\ell^2(\mathbb{R}).
\end{equation*}
We also set  $P_{\infty}=Id$, hence $P_{\infty}v=v$ for all $v\in\ell^2(\mathbb{R})$. For brevity, in this paper we write $\|\cdot\|_{\ell^2}$ instead of $\|\cdot\|_{\ell^2(\mathbb{R})}.$

For vectors $v = [v_1, \ldots, v_m] \in \R^m,$ $u = [u_1, \ldots, u_l] \in \R^l,$ we denote by $v\oplus u$ the vector $[v_1, \ldots, v_m, u_1,\ldots, u_l]\in\mathbb{R}^{m+l}$. By $\bigoplus_{k=1}^n w_k$ we understand the vector $w_1 \oplus w_2 \oplus \ldots \oplus w_n.$

In this paper, we use the following asymptotic notation. For two real-valued sequences $(a_n)_{n=1}^\infty,\, (b_n)_{n=1}^\infty,$ we write $a_n \lessapprox b_n, n\to +\infty,$ if and only if $\limsup_{n\to+\infty} a_n/b_n \leq 1.$ We also say that $a_n \approx b_n, n \to +\infty,$ if and only if $\lim_{n\to +\infty}a_n / b_n = 1.$ Furthermore, the asymptotic symbols $\Omega, \Theta, \mathcal{O},o$ appearing in this paper are aligned with classical Landau notation for sequences. For a sequence $(c(n))_{n=1}^\infty$ of non-negative numbers converging to zero and $\epsilon>0,$ we define the inverse $c^{-1}(\epsilon) = \sup\{n \in \mathbb{N}\, | \, c(n)>\epsilon\}.$
\medskip

We assume that drift coefficient $a: [0,T]\times \mathbb{R} \mapsto \mathbb{R} $ belongs to $\mathcal{C}^{1,2}([0,T]\times \mathbb{R})$ and satisfies the following conditions:
\begin{enumerate}
	\item [(A1)] $|a(t,x) - a(s,x)| \leq C_1(1+ |x|)|t-s|$ for all $t,s\in [0,T], \ x\in \mathbb{R}$,
	\item [(A2)] $|a(t,0)|\leq C_1$ for all $t\in[0,T]$,
	\item [(A3)] $|a(t,x) - a(t,y)| \leq C_1|x-y|$ for all $x,y \in \mathbb{R}$, $t \in [0,T]$,
	\item [(A4)] $|\frac{\partial a}{\partial x}(t,x) - \frac{\partial a}{\partial x}(t,y)| \leq C_1|x-y|$ for all pairs $(t,x), (t,y) \in [0,T]\times \mathbb{R}$
\end{enumerate}
for some $C_1 > 0.$
\smallskip

Let $\delta = (\delta(k))_{k = 1}^{\infty}\subset \mathbb{R}$ be a positive, strictly decreasing sequence vanishing at infinity. For fixed $\delta,$ by $\mathcal{G}_\delta$ we denote a set of all non-decreasing sequences $G=(G(n))_{n=1}^{\infty} \subset \mathbb{N}$ such that $G(n)\to +\infty$ and
\begin{equation}\label{eq:G_delta}
    \lim\limits_{n \to +\infty}n^{1/2}\,\delta\big(G(n)\big) = 0.
\end{equation}

We assume that diffusion coefficient $\sigma = (\sigma_1, \sigma_2, \ldots):[0,T] \mapsto\ell^2(\mathbb{R})$ satisfies the following conditions:
\begin{enumerate}
    \item [(S1)] $\|\sigma(0)\|_{\ell^2} \leq C_2,$
	\item [(S2)] $\Vert \sigma(t) - \sigma(s)\Vert_{\ell^2} \leq C_2|t-s|$ for all $t,s\in [0,T]$,
	\item [(S3)] $\Vert \sigma(t) - P_k \sigma(t)\Vert_{\ell^2} \leq C_2\delta(k)$ for all $k \in \mathbb{N}, \ t \in [0,T],$
\end{enumerate}
for $C_2 > 0$ and some fixed sequence $\delta$ as above.

Our idea is to first provide the approximation of truncated solution $X^M = X^M(a,\sigma,W)$ which depends on the first $M\in \mathbb{N}$ coordinates of the underlying Wiener process $W$. Then, we estimate globally the inevitable truncation error resulting from substituting the process $X$ for $X^M.$ For convenience, we will use the notation $X^{\infty}:= X.$

To this end, we consider the family of processes $X^M, \ M \in \mathbb{N} \cup \{\infty\},$ satisfying
\begin{equation}
\label{findim_equation}
	\left\{ \begin{array}{ll}
	\displaystyle{
	\rd X^M(t) = a(t,X^M(t))\rd t + \sigma^M(t) \rd W(t) , \ t\in [0,T]},\\
	^M(0)=x_0.
	\end{array}\right.
\end{equation}
In particular, for $M=+\infty$ we obtain the main problem \eqref{main_equation}. For further analysis, we need some properties of the process $X^M.$ Those are presented below, in a corollary from Lemma 1 in \cite{PPMSLS}.
\begin{lemma}
\label{lemma_sol}
	For every $M \in \mathbb{N} \cup \{\infty\}$ the equation \eqref{findim_equation} admits a unique strong solution $X^M = (X^M(t))_{t\in [0,T]}.$ Moreover, there exists $K \in (0,+\infty),$ depending only on the constants $C_1, C_2,$ such that for every
	$M\in\mathbb{N}\  \cup \ \{\infty\}$ we have that
	\begin{equation}
	\label{lemma_sol_bound}
		\mathbb{E}\Bigl(\sup_{0 \leq t \leq T} |X^M(t)|^2\Bigr) \leq K
	\end{equation}
	and for all $s,t \in [0,T]$ the following holds 
	\begin{equation*}\label{lemma_sol_regularity}
	    \mathbb{E}|X^M(t)-X^M(s)|^2 \leq K|t-s|.
	\end{equation*}
\end{lemma}
\smallskip


We also state truncation error bound for our model. This result is a corollary from Proposition 1 in \cite{PPMSLS}.
\begin{proposition}
\label{prop:infdim}
  There exists $K_1\in (0,+\infty)$ such that for any $M \in \mathbb{N}$  it holds
 \begin{equation}
     	\sup\limits_{0 \leq t \leq T}\Vert X(a,\sigma, W)(t)-X^M(a,\sigma,W)(t)\Vert_{L^2(\Omega)} \leq K\delta(M).
 \end{equation}
\end{proposition}

We define truncated dimension time-continuous Euler algorithm $\euler = (\euler(t))_{t\in [0,T]}$ that approximates the process $(X(t))_{t \in [0,T]}$. Take $M,n \in \mathbb{N}$, and let
\begin{equation}\label{partition}
    \tau_n: 0 = t_{0,n} < t_{1,n} < \ldots < t_{k_n-1,n} < t_{k_n,n} = T
\end{equation}
be a sequence of partitions of the interval $[0,T],$ with $k_n \in \mathbb{N}.$ 

We set
\begin{equation}\label{continuous_scheme}
	\begin{cases}
		\euler(0) = x_0 \\
		\euler(t) = \euler(t_{j,n}) + a\big(t_{j,n} , \euler(t_{j,n})\big)(t - t_{j,n}) + \sigma^M(t_{j,n})\big(W(t) - W(t_{j,n})\big), \\ 
		\quad\quad\quad\quad\quad t \in [t_{j,n}, t_{j+1,n}],
		\ \ j=0,1,\ldots, k_n-1.
	\end{cases}
\end{equation}
We stress that $\euler$ is not implementable since it requires complete knowledge of the trajectories of the underlying Wiener process.

Now we state the upper error bound of the truncated dimension time-continuous Euler algorithm in finite dimensional setting.
\begin{proposition}\label{prop:findim}
Under the assumptions (A1)-(A4) and (S1)-(S3), there exists a positive constant $C_0$, depending only on $C_1, C_2$, such that for all $M,n\in\mathbb{N}$ the time-continuous truncated Euler process \eqref{continuous_scheme} based on partition \eqref{partition} satisfies
	\begin{equation}\label{Euler_error_upper_bound}
		\sup_{0 \leq t \leq T} \Vert\euler(t)-X^M(t)\Vert_{L^2(\Omega)} \leq C_0 \max_{0 \leq j \leq k_n-1}\Delta t_{j,n},
	\end{equation}
	where $\Delta t_{j,n} = t_{j+1,n} - t_{j,n},$ $j=0,1,\ldots, k_n-1.$
\end{proposition}
\noindent The proof of Proposition \ref{prop:findim} is postponed to the Appendix.
\medskip

From Proposition \ref{prop:infdim} and Proposition \ref{prop:findim} we obtain the following result.
\begin{theorem}\label{upper_bound}
	 Let the coefficients $a,\sigma$ satisfy (A1)-(A4) and (S1)-(S3) with sequence $\delta$, respectively. Then, there exists a positive constant $K$, depending on $C_1, C_2, \delta,$ such that for every $M,n\in\mathbb{N}$ and the discretisation \eqref{partition}, it holds
	\begin{equation*}
	    \sup_{0\leq t \leq T}\|X(t) - \euler(t)\|_{L^2(\Omega)}\leq K\Bigl(\max\limits_{0\leq j \leq k_n-1} \,\Delta t_{j,n} +  \delta(M)\Bigr).
	\end{equation*}
\end{theorem}

At the end of this section, we present several remarks on the model, imposed assumptions, and suitability of the chosen stochastic scheme.

\begin{remark}
Generally, the concept of SDEs with integrals wrt countably dimensional Wiener process is used to describe the evolution driven by countably many risk factors. This modelling choice can be leveraged for a wide range of problems in, e.g., genetics, mathematical finance or physics \cite{carmona, PBL}. From a pragmatic point of view, infinite dimensional setting can be leveraged when the number of random risks is finite but still too large to be entirely captured by the electronic machines. On the other hand, the integrals appearing in this paper can be viewed as stochastic integrals with respect to cylindrical Brownian motion in Hilbert space of sequences $\ell^2(\mathbb{R})$ and hence, our model forms a bridge between ordinary SDEs and stochastic partial differential equations (SPDEs), see \cite{zabczyk}.
\end{remark}

\begin{remark}
    First, the assumptions (A1) - (A4) imply the existence of a constant $K_0 > 0$ such that 
    \begin{equation}\label{eq:bounded_derivatives}
    \max\bigg\{\left|\frac{\partial a}{\partial x}(t,x)\right|, \left|\frac{\partial^2 a}{\partial x^2}(t,x)\right|\bigg\} \leq K_0 \quad \textrm{for every} \quad (t,x)\in [0,T] \times \mathbb{R}.
    \end{equation}
Second, in the presented setting, a crucial role is played by appropriate choice of corresponding truncation levels for admissible methods. Indeed, those are defined in terms of the elements of $\mathcal{G}_{\delta}$ as per \eqref{eq:G_delta}. We note that for every $\delta$ the corresponding set $\mathcal{G}_\delta$ is nonempty. 
    Furthermore, the slower rate of diffusion decay to zero, the greater values of $G(n)$ need to be taken. We note that $\delta$ does not need to be optimal in a sense that for fixed $\sigma$ there might exist $\delta'$ also satisfying (S3) and $\delta' < \delta.$ Nevertheless, sharper bound in (S3) yields greater palette of the corresponding sequences $G$ in \eqref{eq:G_delta}, as $\delta' \leq \delta$ implies $\mathcal{G}_\delta \subset \mathcal{G}_{\delta'}.$
\end{remark}

\begin{remark}
It is worth mentioning that suitable modifications of the classic Euler scheme for finite-dimensional setting have been investigated in e.g. \cite{Muller, PMPP14, PPMSLS}. For the problem \eqref{main_equation}, the method $\euler$ coincides with truncated dimension time-continuous Euler algorithm proposed in \cite{PPMSLS}. However, the regularity of function $a$ in our case enhances the rate of convergence from $1/2$ to $1,$ see Theorem \ref{upper_bound}. Indeed, the proposed method also coincides with the Milstein scheme; see e.g., \cite{AKPP, KRWU, PMPP19, PP1, PP2, PPVSMS, PPVSMS2} where the approximation by modified Milstein schemes for finite dimensional models was considered.
\end{remark}

In the following sections, we investigate lower bounds for exact asymptotic error behaviour and construct optimal methods in suitable subclasses of admissible algorithms. The optimality is defined in the spirit of Information-Based Complexity (IBC) framework, see also \cite{TWW88} for more details.
\section{Lower bounds for exact asymptotic error behaviour}\label{sec:lower_bound}
In this section, we derive minimal global approximation error for our initial problem \eqref{main_equation}. The main results are presented in Theorem \ref{th:AC_lower}.

Let us fix $(a,\sigma),$ and a sequence $\delta$ satisfying (S3). An arbitrary method under consideration is a sequence of the form $\overline{X} = (\xalg)_{n=1}^{\infty}$ and can be equivalently viewed as a quadruple $\overline{X} = (\bar{\Delta}, \bar{\mathcal{N}}, \bar{M}, \bar{\phi}),$ where:
\begin{itemize}
    \item $\bar{\Delta}=(\bar{\Delta}_{n})_{n=1}^{\infty}$ is a sequence of (possibly) non-expanding partitions of the interval $[0,T],$ i.e.,
    \begin{equation}\label{arbitary_partition}
        \bar{\Delta}_{n}: \quad 0=\bar{t}_{0,n} < \bar{t}_{1,n} < \ldots < \bar{t}_{\bar{k}_n-1,n} < \bar{t}_{\bar{k}_n,n} = T,
    \end{equation}
    where for some $\overline{C}_1, \overline{C}_2>0$ it holds
    \begin{equation}\label{eq:bar_kn}
    \overline{C}_1\, n \geq \bar{k}_n \geq \overline{C}_2 \, n^{1/2}, \quad n\geq n_0(\overline{X}) \in \mathbb{N}.
    \end{equation}
    \item $\bar{\mathcal{N}} = (\mathcal{N}_{M_n, \bar{k}_n})_{n=1}^{\infty}$ is a sequence of information vectors. For fixed $n \in \mathbb{N},$ the vector $\mathcal{N}_{M_n, \bar{k}_n}$  consists of the points in which the method $\xalg$ evaluates the values of underlying (scalar) Wiener processes $W_k,$ $k=1,\ldots, M_n:$
    \begin{equation*}
        \mathcal{N}_{M_n, \bar{k}_n} = \bigoplus_{k=1}^{M_n}\,\big[W_k(\bar{t}_{1,n}), W_k(\bar{t}_{2,n}),\ldots, W_k(\bar{t}_{\bar{k}_n,n})\big].
    \end{equation*}
    \item $\bar{M} = (M_n)_{n=1}^{\infty} \in \mathcal{G}_\delta$ with $M_n$ indicating number of initial Wiener process coordinates used by the method $\xalg, \,n \in \N$.
    \item The method $\xalg$ is assumed to evaluate $W$ in the time points from $\mathcal{N}_{M_n, \bar{k}_n},$ yielding a process $\xalg(a,\sigma,W)$ being an approximation of $X.$ Namely, we assume the existence of Borel measurable mappings $\bar{\phi} = (\phi_n)_{n=1}^\infty,$ with $\phi_n: \mathbb{R}^{\bar{k}_n \cdot M_n} \mapsto L^2([0,T]),$ such that
    \begin{equation*}
    \phi_n(\mathcal{N}_{M_n, \bar{k}_n}(W)) = \xalg, \quad n \in \mathbb{N}.
    \end{equation*}
\end{itemize}
\medskip

The class of algorithms satisfying above conditions is denoted by $\chi_{noneq}.$ In this paper, we distinguish a subclass $\chi_{eq}\subset \chi_{noneq}$ of methods leveraging equidistant partitions
\begin{equation}\label{def:chieq}
    \chi_{eq} = \big\{\overline{X}\in\chi_{noneq}\,|\, \exists_{n_0 = n_0(\overline{X})}: \forall_{n\geq n_0} \ \bar{\Delta}_n = \{jT/\bar{k}_n \ : \ j=0,1,\ldots, \bar{k}_n\}\big\}.
\end{equation}

The optimality in class of methods leveraging equidistant meshes for global approximation problem in finite dimensional model has been considered recently in e.g., \cite{AK22}. In this paper we also show the benefit of leveraging adaptive meshes instead of equidistant ones.

By the cost of the algorithm $\xalg$ we understand a number of evaluations of scalar Wiener processes performed by $\xalg$. Specifically, we have
$$\cost(\xalg) = \begin{cases}
    M_n \cdot \bar{k}_n,  \quad \quad \textrm{when  } \sigma \not\equiv 0, \\
    0, \quad\quad\quad\quad \ \ \textrm{when  } \sigma \equiv 0.
\end{cases}
$$
While in case of $\sigma \equiv 0$ we actually deal with ordinary differential equations and still some calculations need to be performed, there is no $W$ process involved. In our setting, this is justified by zero cost.

The global approximation error is measured in a product $L^2(\Omega \times [0,T])$ norm 
\begin{equation*}
    \big\|X - \xalg\big\|_2 := \bigg(\mathbb{E}\int_{0}^T|X(t) - \xalg(t)|^2 \rd t\bigg)^{1/2}.
\end{equation*}

For a fixed method $\xalg=(\bar{\Delta}, \bar{\mathcal{N}}, \bar{M}, \bar{\phi}) \in \chi_{noneq}\setminus \chi_{eq}$ we define the sequence of augmented partitions $\widetilde{\Delta} = (\widetilde{\Delta}_n)_{n=1}^\infty$ with
\begin{equation}\label{eq:bar_delta_partition}
    \widetilde{\Delta}_n := \bar{\Delta}_n \,\cup\, \Delta_{m_n}^{eq} , \quad n\in\N,
\end{equation}
where $\bar{k}_n^{1/2}/ m_n \to 0$ and $m_n/ \bar{k}_n \to 0,$ $n\to+\infty,$ and $\Delta_{m_n}^{eq} = (t_j^{eq})_{j=0}^{m_n}$ is an equidistant partition, $t_j^{eq} = Tj/m_n, \ j=0, \ldots, m_n.$ In the sequel, by $\widetilde{k}_n$ we denote the number of distinct time points $\widetilde{t}_{j,n}$ in $\widetilde{\Delta}_n.$ Consequently, we get 
\begin{equation*}
\bar{k}_n \leq \widetilde{k}_n \leq \bar{k}_n + m_n -1,
\end{equation*}
which in turn implies that $\lim_{n\to +\infty}\widetilde{k}_n / \bar{k}_n = 1.$
In addition, we introduce augmented information vectors $\widetilde{\mathcal{N}} = (\widetilde{\mathcal{N}}_{M_n, \widetilde{k}_n}(W))_{n=1}^{\infty},$ where
\begin{equation*}
    \widetilde{\mathcal{N}}_{M_n, \widetilde{k}_n}(W) = \bigoplus_{k=1}^{M_n}\,\big[W_k(\widetilde{t}_{0,n}), W_k(\widetilde{t}_{1,n}),\ldots, W_k(\widetilde{t}_{\,\widetilde{k}_n,n})\big].
\end{equation*}

For brevity, in the sequel we will use the notation $\Delta \widetilde{t}_{j,n} = \widetilde{t}_{j+1,n} - \widetilde{t}_{j,n}, \ n\in\N, \ j = 0,1,\ldots, \widetilde{k}_n-1.$

Now for fixed $n\in \N$ we estimate distance between truncated dimension time-continuous Euler process $\widetilde{X}_{M_n,\widetilde{k}_n}^{E}$ based on partition $\widetilde{\Delta}_{n}$ and the associated time-continuous conditional Euler process
$$X_{M_n, k_n}^{cond}(t) = \mathbb{E}\big(\widetilde{X}_{M_n,\widetilde{k}_n}^{E}(t) \big|\,\widetilde{\mathcal{N}}_{M_n,\widetilde{k}_n}(W)\big),$$
which leverages the augmented information $\widetilde{\mathcal{N}}_{M_n, \widetilde{k}_n}(W).$ We refer to, e.g., \cite{AKPP} for more details on conditional Euler process in finite dimensional setting when also jumps modelled by homogeneous Poisson process are considered.

We obtain
\begin{equation}\label{eq:euler_minus_cond}
    \begin{split}
        \big\|\widetilde{X}_{M_n,\widetilde{k}_n}^{E} - \mathbb{E}\big(\widetilde{X}_{M_n,\widetilde{k}_n}^{E} \big|\, \widetilde{\mathcal{N}}_{M_n,\widetilde{k}_n}(W)\big)\big\|_2^2 & = \mathbb{E}\int_0^T\sum_{j=0}^{\widetilde{k}_n-1}\sum_{k=1}^{M_n} \bigg(W_k(t) - \mathbb{E}\big(W_k(t)\big|\,\widetilde{\mathcal{N}}_{M_n,\widetilde{k}_n}(W)\big)\bigg)^2 \\
        & \quad\quad\quad \times(\sigma_k(\widetilde{t}_{j,n}))^2\one_{(\widetilde{t}_{j,n}, \widetilde{t}_{j+1,n}]}(t)\rd t \\
        & = \sum_{j=0}^{\widetilde{k}_n -1} \int_{\widetilde{t}_{j,n}}^{\widetilde{t}_{j+1,n}}\sum_{k=1}^{M_n} (\sigma_k(\widetilde{t}_{j,n}))^2  \ \mathbb{E}\big(\hat{W}_{j,k,n}(t)\big)^2 \rd t,
    \end{split}
\end{equation}
where $\hat{W}_{j,k,n}$ is a Brownian bridge on the interval $[\widetilde{t}_{j,n}, \widetilde{t}_{j+1,n}],$ conditioned on $W_k.$ For more details on Brownian bridge, we refer to e.g. \cite{PP1}. Furthermore,
\begin{equation}\label{eq:conditional_error}
    \begin{split}
        \big\|\widetilde{X}_{M_n,\widetilde{k}_n}^{E} - \mathbb{E}\big(\widetilde{X}_{M_n,\widetilde{k}_n}^{E} \big| \widetilde{\mathcal{N}}_{M_n,\widetilde{k}_n}(W)\big)\big\|_2^2& = \sum_{j=0}^{\widetilde{k}_n -1} \int_{\widetilde{t}_{j,n}}^{\widetilde{t}_{j+1,n}} \frac{(\widetilde{t}_{j+1,n}-t)(t-\widetilde{t}_{j,n})}{\Delta \widetilde{t}_{j,n}}\sum_{k=1}^{M_n} (\sigma_k(\widetilde{t}_{j,n}))^2 \rd t \\
        & = \frac{1}{6}\sum_{j=0}^{\widetilde{k}_n-1}\|\sigma^{M_n}(\widetilde{t}_{j,n})\|_{\ell^2}^2(\Delta \widetilde{t}_{j,n})^2.
    \end{split}
\end{equation}

Combining Proposition \ref{prop:infdim}, \eqref{eq:conditional_error}, and the fact that $\xalg$ is $\sigma(\bar{\mathcal{N}}_{M_n,\bar{k}_n}(W))$-measurable, we get
\begin{equation}\label{eq:err_lower_est_1}
    \begin{split}
        \big\|X - \xalg\big\|_2 &\geq \big\|\xalg - \widetilde{X}_{M_n,\widetilde{k}_n}^{E}\big\|_2 - \big\|X - \widetilde{X}_{M_n,\widetilde{k}_n}^{E}\big\|_2 \\
        &\geq \big\|\widetilde{X}_{M_n,\widetilde{k}_n}^{E} - \mathbb{E}\big(\widetilde{X}_{M_n,\widetilde{k}_n}^{E} \big|\, \bar{\mathcal{N}}_{M_n,\bar{k}_n}(W)\big)\big\|_2 - D_1m_n^{-1} - D_2\delta(M_n) \\
        &\geq \big\|\widetilde{X}_{M_n,\widetilde{k}_n}^{E} - \mathbb{E}\big(\widetilde{X}_{M_n,\widetilde{k}_n}^{E} \big|\, \widetilde{\mathcal{N}}_{M_n,\widetilde{k}_n}(W)\big)\big\|_2 - D_1m_n^{-1} - D_2\delta\big(M_n\big) \\
        & = \biggr(\frac{1}{6}\sum_{j=0}^{\widetilde{k}_n-1}\|\sigma^{M_n}(\widetilde{t}_{j,n})\|_{\ell^2}^2(\Delta \widetilde{t}_{j,n})^2\biggr)^{1/2} - D_1m_n^{-1} - D_2\delta\big(M_n\big),
    \end{split}
\end{equation}
where both $D_1, D_2$ do not depend on $n$ (explicitly or implicitly via $M_n$ or $\bar{k}_n, \widetilde{k}_n$).
Above, we also leveraged property $\bar{\mathcal{N}}_{M_n,\bar{k}_n}(W) \subset \widetilde{\mathcal{N}}_{M_n,\widetilde{k}_n}(W).$ 
Hence, by \eqref{eq:G_delta}, \eqref{eq:bar_kn}, \eqref{eq:err_lower_est_1}, and the definition of $m_n,$ we have
\begin{equation}\label{eq:err_lower_est_2}
    \begin{split}
        \liminf_{n \to +\infty} (M_n)^{-1/2}\big(\cost(\xalg)\big)^{1/2}\big\|X - \xalg\big\|_2 & \geq \liminf\limits_{n \to +\infty}\biggr(\frac{\bar{k}_n}{6}\sum_{j=0}^{\widetilde{k}_n -1}\|\sigma^{M_n}(\widetilde{t}_{j,n})\|_{\ell^2}^2 (\Delta \widetilde{t}_{j,n})^2\bigg)^{1/2} \\
        & - \limsup\limits_{n\to +\infty}\,\bar{k}_n^{1/2}\big(D_1m_n^{-1} + D_2\,\delta(M_n)\big) \\
        & \hspace{-0.5cm}= \liminf\limits_{n \to +\infty}\biggr(\frac{\bar{k}_n}{6}\sum_{j=0}^{\widetilde{k}_n -1}\|\sigma^{M_n}(\widetilde{t}_{j,n})\|_{\ell^2}^2 (\Delta \widetilde{t}_{j,n})^2\bigg)^{1/2}.
    \end{split}
\end{equation}
By Fact \ref{fact:integral}, \eqref{eq:err_lower_est_2}, the H\"{o}lder inequality, and the fact that $\widetilde{k}_n \approx \bar{k}_n, \ n\to +\infty,$ we arrive at
\begin{equation*}
    \begin{split}
        \liminf_{n \to +\infty} (M_n)^{-1/2}\big(\text{cost}(\xalg)\big)^{1/2}\big\|X - \xalg\big\|_2 & \geq \frac{1}{\sqrt{6}}\int_{0}^T \|\sigma(t)\|_{\ell^2}\rd t.
    \end{split}
\end{equation*}
Since
\begin{equation*}
        \liminf_{n \to +\infty} \frac{\big(\text{cost}(\xalg)\big)^{1/2}\big\|X - \xalg\big\|_2}{M_n^{1/2}\mathcal{C}_{noneq}} = \left(\limsup_{n \to +\infty} \frac{M_n^{1/2}\mathcal{C}_{noneq}}{\big(\text{cost}(\xalg)\big)^{1/2}\big\|X - \xalg\big\|_2}\right)^{-1},
\end{equation*}
we finally get for all $\xalg \in \chi_{noneq}$ that
\begin{equation}\label{eq:proof_lower_bound_noneq}
        \big(\text{cost}(\xalg)\big)^{1/2}\,\big\|X - \xalg\big\|_2 \gtrapprox M_n^{1/2}\,\mathcal{C}_{noneq}, \quad n\to +\infty.
\end{equation}
\medskip

For $\xalg \in \chi_{eq},$ the proof follows analogous steps, except for taking simply $\widetilde{\Delta}_n = \bar{\Delta}_n$ in \eqref{eq:bar_delta_partition}. Leveraging the fact that $\Delta \widetilde{t}_{j,n} = \frac{T}{\bar{k}_n}, \ j=0,\ldots, \bar{k}_n -1,$ leads to
\begin{equation}\label{eq:err_lower_est_3}
    \begin{split}
        \liminf_{n \to +\infty} (M_n)^{-1/2}\big(\cost(\xalg)\big)^{1/2}\big\|X - \xalg\big\|_2 & \geq
        \liminf\limits_{n \to +\infty}\biggr(\frac{1}{6}\sum_{j=0}^{\widetilde{k}_n -1}\|\sigma^{M_n}(\widetilde{t}_{j,n})\|_{\ell^2}^2 \Delta \widetilde{t}_{j,n}\bigg)^{1/2}.
    \end{split}
\end{equation}
Consequently, \eqref{eq:err_lower_est_3} implies that for all $\xalg \in \chi_{eq}$ it holds
\begin{equation}\label{eq:proof_lower_bound_eq}
        \big(\text{cost}(\xalg)\big)^{1/2}\,\big\|X - \xalg\big\|_2 \gtrapprox M_n^{1/2}\,\mathcal{C}_{eq}, \quad n\to +\infty.
\end{equation}
Considering \eqref{eq:proof_lower_bound_noneq} and \eqref{eq:proof_lower_bound_eq}, we are ready to formulate the main result of this section.
\medskip

\begin{theorem}\label{th:AC_lower}
Let us denote by $\chi_{noneq}^{\bar{M}}\subset \chi_{noneq}$ a set of all admissible methods with fixed truncation level sequence $\bar{M}.$ We have the following asymptotic lower bound 
\begin{equation}\label{th:AC_noneq}
        \inf\limits_{\overline{X} \in \chi_{noneq}^{\bar{M}}}\big(\cost(\xalg)\big)^{1/2}\big\| X - \xalg\big\|_2  \gtrapprox \frac{M_n^{1/2}}{\sqrt{6}}\int_{0}^T \|\sigma(t)\|_{\ell^2}\rd t.
\end{equation}
In particular, restriction to the subclass $\chi_{eq}$ gives sharper asymptotic lower bound
\begin{equation}\label{th:AC_eq}
        \inf\limits_{\overline{X} \in \chi_{noneq}^{\bar{M}}\,\cap\,\chi_{eq}}\big(\cost(\xalg)\big)^{1/2}\big\|X - \xalg \big\|_{2} \gtrapprox \sqrt{\frac{M_n T}{6}}\biggr(\int_{0}^T \|\sigma(t)\|_{\ell^2}^2\rd t \biggr)^{1/2}.
\end{equation}
\end{theorem}
\medskip

\begin{remark}
The restriction \eqref{eq:bar_kn} imposed on $\bar{k}_n$ is not limiting, as the majority of algorithms used in practice leverage $\mathcal{O}(n)$ nodes. We also stress that $\overline{C}_1, \overline{C}_2,$ as well as the related sequence $\bar{k}_n,$ depend on the method $\overline{X}$ and need not be the same across all considered algorithms. Furthermore, we allow only discrete, finite-dimensional evaluations of $W.$ Therefore, for all $n\in\mathbb{N},$ the vector $\mathcal{N}_{M_n, \bar{k}_n}$ is $\sigma\big(W_1, W_2,\ldots,W_{M_n}\big)$-measurable. Moreover, different partitions among various coordinates of $W$ are not permitted in our setting.
\end{remark}

\begin{remark}
We stress that while the sequences $\bar{\Delta}$ and $\bar{\phi}$ might depend on $a,\sigma,$ they cannot base on the trajectory of $W.$ The resulting information about Wiener process is then called non-adaptive, while the associated algorithms are referred to as path-independent. For path-dependent version of Euler algorithm in finite dimensional setting, see e.g., \cite{PP2}. 
\end{remark}

\begin{remark}
It should be noted that lower bounds in Theorem \ref{th:AC_lower} diverge to infinity when $n \to +\infty$ in both subclasses. This significantly differs from finite-dimensional case, when the truncation level sequence was bounded from above. That saying, when $M_n \equiv 1,$ we have $\|\sigma(t)\|_{\ell^2}^2 = |\sigma_1(t)|^2$ for all $t\in [0,T],$ and the asymptotic constant appearing in \eqref{th:AC_noneq} is consistent with the result from Theorem 2 in \cite{Muller}.
\end{remark}

Notably, for given sequence $\overline{M},$ in Section \ref{sec:equidistant} and Section \ref{sec:step_size} we construct algorithms which asymptotically attain lower bounds appearing in Theorem \ref{th:AC_lower}. Then, in Section \ref{sec:optimal_alg} we discuss the existence of optimal algorithms in a class of methods taking into account all permitted truncation sequences $\overline{M}.$

\section{Construction of asymptotically optimal methods}\vspace{0.2cm}

\subsection{Optimal algorithm in class of methods \texorpdfstring{$\chi_{eq}^{\bar{M}}$}{}}\label{sec:equidistant}

\phantom{a}\vspace{0.17cm}\newline
First, we fix $\bar{M} = (M_n^*)_{n=1}^{\infty}$ satisfying \eqref{eq:G_delta}.
Next, for any $n \in \mathbb{N}$ let us denote by $X_{M_n^*, n}^{Eq}\in \chi_{eq}$ the truncated dimension Euler method based on equidistant mesh $\Delta^{eq}_n: 0 = t_{0,n}^{eq} < \ldots <~t_{n,n}^{eq} =~T,$ where $t_{j,n}^{eq} = \frac{Tj}{n}, \ j=0,1,\ldots, n.$ The scheme is defined as follows:
\begin{equation}\label{euler_eq_scheme}
	\begin{cases}
		X_{M_n^*, n}^{Eq}(0) = x_0 \\
		X_{M_n^*, n}^{Eq}(t) = X_{M_n^*, n}^{Eq}(t_{j,n}^{eq}) + a\big(t_{j,n}^{eq} , X_{M_n^*, n}^{Eq}(t_{j,n}^{eq})\big)(t_{j+1,n}^{eq} - t_{j,n}^{eq}) \\
        \hspace{2.0cm} + \sigma^{M_n^*}(t_{j,n}^{eq})\big(W(t_{j+1,n}^{eq}) - W(t_{j,n}^{eq})\big), \quad\quad t \in [t_{j,n}, t_{j+1,n}],
		\ \ j=0,1,\ldots, n-1.
	\end{cases}
\end{equation}
The outcome of the method is a stochastic process $\big(X_{M_n^*, n}^{Eq*}(t)\big)_{t\in [0,T]}$ obtained by linear interpolation between two subsequent nodes $t_{j,n}^{eq}$ and $t_{j+1,n}^{eq},$ i.e., 
\begin{equation}
    X_{M_n^*,n}^{Eq*}(t) = \frac{X_{M_n^*,n}^{Eq}(t_{j,n}^{eq})(t_{j+1,n}^{eq} - t) + X_{M_n^*,n}^{Eq}(t_{j+1,n}^{eq})(t - t_{j,n}^{eq})}{t_{j+1,n}^{eq} - t_{j,n}^{eq}}, \quad t\in \big[t_{j,n}^{eq}, t_{j+1,n}^{eq}\big],
\end{equation}
$j=0,1,\ldots, n-1.$ Let us also denote by $\mathcal{N}_{M_n^*,n}^{eq}(W)$ the related information vector. In addition, by $\widetilde{X}_{M_n^*,n}^{E,eq}$ we understand the truncated dimension time-continuous Euler process based on the mesh $\Delta^{eq}_n,$ $n\in \mathbb{N}.$
Certainly, it holds
\begin{equation}\label{eq:equiv_optimal_1}
    n^{1/2}\bigg|\big\|X - X_{M_n^*,n}^{Eq*}\big\|_2 - \big\|\widetilde{X}_{M_n^*,n}^{E,eq} - X_{M_n^*,n}^{Eq*}\big\|_2\bigg| \leq n^{1/2}\, \big\|X - \widetilde{X}_{M_n^*,n}^{E,eq}\big\|_2.
\end{equation}
Consequently, from Theorem \ref{upper_bound} and \eqref{eq:equiv_optimal_1} it follows
\begin{equation}\label{eq:equiv_optimal_2}
    \begin{split}
    n^{1/2}\bigg|\big\|X - X_{M_n^*,n}^{Eq*}\big\|_2 - \big\|\widetilde{X}_{M_n^*,n}^{E,eq} - X_{M_n^*,n}^{Eq*}\big\|_2\bigg|
    \leq \frac{KT}{n^{1/2}} + Kn^{1/2}\delta(M_n^*).
    \end{split}
\end{equation}
Therefore, by repeating the reasoning as in \eqref{eq:conditional_error} and the fact that
\begin{equation*}
X_{M_n^*,n}^{Eq*} = \mathbb{E}\big(\widetilde{X}_{M_n^*,n}^{E,eq}\,|\,\mathcal{N}_{M_n^*,n}^{eq}(W)\big),
\end{equation*}
the inequality \eqref{eq:equiv_optimal_2} implies
\begin{equation}\label{eq:eq_limit_final}
    \begin{split}
    \lim\limits_{n\to+\infty}n^{1/2}\big\|X - X_{M_n^*,n}^{Eq*}\big\|_2 &= \lim\limits_{n\to+\infty}n^{1/2}\big\| \widetilde{X}_{M_n^*,n}^{E,eq} - \mathbb{E}\big(\widetilde{X}_{M_n^*,n}^{E,eq}\,|\, \mathcal{N}_{M_n^*,n}^{eq}(W)\big)\big\|_2 \\
    &= \lim\limits_{n\to+\infty}\biggr(\frac{T}{6}\sum_{j=0}^{n -1}\|\sigma^{M_n^*}(t_{j,n}^{eq})\|_{\ell^2}^2 \,\frac{T}{n}\bigg)^{1/2} \\
    & = \sqrt{\frac{T}{6}}\biggr(\int_{0}^T \|\sigma(t)\|_{\ell^2}^2\rd t \biggr)^{1/2}.
    \end{split}
\end{equation}
Finally, by \eqref{eq:eq_limit_final} we have
\begin{equation}
    \big(\cost(X_{M_n^*,n}^{Eq*})\big)^{1/2}\,\big\|X - X_{M_n^*,n}^{Eq*}\big\|_2 \approx \sqrt{\frac{M_n^*T}{6}}\biggr(\int_{0}^T \|\sigma(t)\|_{\ell^2}^2\rd t \biggr)^{1/2},
\end{equation}
which establishes the lower bound in Theorem \ref{th:AC_lower} for fixed class $\chi_{eq}^{\bar{M}}$. Note that $X_{M_n^*,n}^{Eq}$ is an implementable algorithm, and it does not require the knowledge of the trajectories of (truncated) Wiener process.
\subsection{Optimal algorithm with adaptive path-independent step-size control in class \texorpdfstring{$\chi_{noneq}^{\bar{M}}$.}{}}\label{sec:step_size}
\phantom{a}\vspace{0.17cm}\newline
In this section, we construct an optimal method in class $\chi_{noneq}^{\bar{M}}$ for fixed truncation level sequence $\bar{M} = (M_n^*)_{n=1}^{\infty}\in \mathcal{G}_\delta$. To this end, we define truncated dimension Euler scheme $(\xstep)_{n=1}^{\infty}$ with adaptive path-independent step-size of the following form.

First, let $\bar{\varepsilon} = (\varepsilon_n)_{n=1}^{\infty} \subset \R_+$ be a non-increasing sequence satisfying
\begin{equation}\label{eq:varepsilon_prop}
\lim\limits_{n \to +\infty}\varepsilon_n = \lim\limits_{n \to +\infty}\frac{1}{n \,\varepsilon_n^2} = 0.
\end{equation}
For fixed $n \in \mathbb{N}$ the proposed scheme utilises step-size control with $\hat{t}_{0,n} := 0$ and
\begin{equation}\label{eq:step_size}
    \hat{t}_{j+1,n} := \hat{t}_{j,n} + \frac{T}{n\max\{\varepsilon_n, \|\sigma^{M_n^*}(\hat{t}_{j,n})\|_{\ell^2}\}}, \quad j=0,1,\ldots, k_n^*-1,
\end{equation}
where $k_n^* = \inf\{j \in \mathbb{N} \ | \ \hat{t}_{j,n} \geq T \},$ $n\in\mathbb{N}.$ We will denote the corresponding mesh by $\hat{\Delta}_n.$ Then, we set
\begin{equation*}
	\begin{cases}
		\xstep(0) = x_0 \\
		\xstep(\hat{t}_{j+1,n}) = \xstep(\hat{t}_{j,n}) + a(\hat{t}_{j,n} , \xstep(\hat{t}_{j,n}))(\hat{t}_{j+1,n} - \hat{t}_{j,n}) \\ \hspace{5.14cm}+ \ \sigma^{M_n^*}(\hat{t}_{j,n})(W(\hat{t}_{j+1,n}) - W(\hat{t}_{j,n})), \ \ 
		\\ \quad\quad\quad\quad j=0,1,\ldots, k_n^*-1.
	\end{cases}
\end{equation*}
In the sequel, we will use the notation $\Delta^*_n = \hat{\Delta}_n \setminus \{\hat{t}_{k_n^*,n}\} \cup \{T\}$ and $t_{j,n}^* = \hat{t}_{j,n}, \ j=0,1,\ldots, \linebreak k_n^*-1,$ $t_{k_n^*,n}^*=T.$ The corresponding information vector will be denoted by $\mathcal{N}_{M_n^*,k_n^*}^*(W).$ Also, by $\Delta \hat{t}_{j,n},$ $\Delta t_{j,n}^*$ we will understand the $j$--th time step for the discretisation $\hat{\Delta}_{n}$ and $\Delta_{n}^*,$ respectively. 
\smallskip

The final process $X^{*}_{M_n^*,k_n^*} = \big(X^{*}_{M_n^*,k_n^*}(t)\big)_{t\in [0,T]}$ approximating $X$ is obtained by piecewise linear interpolation between $\xstep({t}_{j,n}^*)$ and $\xstep({t}_{j+1,n}^*).$
We recall that in our model~\eqref{main_equation}, the process $X^{*}_{M_n^*,k_n^*}$ coincides with the time-continuous conditional Euler process
\begin{equation}\label{eq:stepsize_vs_conditional}
    \widetilde{X}_{M_n^*,k_n^*}^{cond*}(t) := \mathbb{E}\big(\,\widetilde{X}_{M_n^*,k_n^*}^{E}(t) \ | \ \mathcal{N}_{M_n^*,k_n^*}^*(W)\big), \quad t\ \in [0,T],
\end{equation}
based on the sequence of partitions $\Delta^*_n.$
    
\begin{fact}\label{fact:stepsize}
\phantom{a}\newline\vspace{-0.61cm}
    \begin{itemize}
        \item [a) ] \hspace{-0.3cm} The proposed method $\xstep$ with adaptive step-size control is an element of $\chi_{noneq}$ and attains point $T,$ irrespective of prior choice of the sequences $\bar{M}$ and $\bar{\varepsilon}.$
        \item [b) ] $k_n^*$ is deterministic and $\lim\limits_{n\to +\infty} k_n^*(\sigma) = +\infty.$
        \item [c) ] \hspace{-0.3cm} $\max\limits_{0\leq j \leq k_n^*-1} (t_{j+1,n}^* - t_{j,n}^*) \leq \frac{T}{n\varepsilon_n} \to 0, \ \ n \to +\infty.$
    \end{itemize}
\end{fact}

\begin{proof}
Since for every $M,n \in \mathbb{N}$ and $j = 0, \ldots, k_n^*$ we have
\begin{equation*}
    \|\sigma^M (\hat{t}_{j,n})\|_{\ell^2} \leq C_2\delta(M) + \|\sigma (\hat{t}_{j,n}) - \sigma(\hat{t}_{0,n})\|_{\ell^2} \leq C\|\delta\|_{\ell^\infty(\mathbb{R})} + C_2T =: \hat{C} < +\infty,
\end{equation*}
for $n_0 = \lfloor n (\varepsilon_n + \hat{C})\rfloor + 1$ it holds $\hat{t}_{n_0,n} \geq T.$ Combining this, the fact that $n\varepsilon_n = o(n),$ and $n\varepsilon_n = \Omega(n^{1/2})$ by \eqref{eq:varepsilon_prop}, we have that $\frac{1}{T}n\varepsilon_n^2 \leq \frac{1}{T}n\varepsilon_n \leq k_n^* \leq n_0(n),$ $n\in\mathbb{N}.$ This proves both assertions a) and~b).
\end{proof}

Now we investigate asymptotic behaviour of the method error. Let us denote
\begin{equation}\label{eq:stepsize_s_star_def}
    \widetilde{S}_{M_n^*,n}^{\,l} :=\sum_{j = 0}^{k_n^* - 1}\max\left\{\varepsilon_n^l, \big\|\sigma^{M_n^*}(t_{j,n}^*)\big\|_{\ell^2}^{l}\right\}(t_{j+1,n}^* - t_{j,n}^*)^l, \quad l\in\{1,2\},
\end{equation}
and
\begin{equation}\label{eq:stepsize_s_dash_def}
    \hat{S}_{M_n^*,n}^{\,l} :=\sum_{j = 0}^{k_n^* - 1}\max\left\{\varepsilon_n^l,\big\|\sigma^{M_n^*}(\hat{t}_{j,n})\big\|_{\ell^2}^l\right\}(\hat{t}_{j+1,n} - \hat{t}_{j,n})^l, \quad l \in \{1,2\}.
\end{equation}
By Fact \ref{fact:integral} we get that
\begin{equation}\label{eq:step_sums1}
    \left|\sum_{j = 0}^{k_n^* - 1} \big\|\sigma^{M_n^*}(t_{j,n}^*)\big\|_{\ell^2}\,\Delta t_{j,n}^* - \int_{0}^T \|\sigma(t)\|_{\ell^2}\,\text{d}t\right| \leq K_1 \left(\frac{T}{n\varepsilon_n} + \delta(M_n^*)\right).
\end{equation}
Consequently, from \eqref{eq:stepsize_s_star_def}, \eqref{eq:step_sums1}, and the fact that for all $a,b \geq 0$
\begin{equation*}
    \frac{a-b}{2} \leq \max\{a,b\} - b \leq a,
\end{equation*}
we obtain
\begin{equation}\label{eq:step_sums2}
\begin{split}
    \biggr|\widetilde{S}_{M_n^*,n}^{\,1} - \int_{0}^T \|\sigma(t)\|_{\ell^2}\,\text{d}t\biggr| & \leq \biggr|\sum_{j=0}^{k_n^* - 1}\bigg(\max\left\{\varepsilon_n,\big\|\sigma^{M_n^*}(t_{j,n}^*)\big\|_{\ell^2}\right\} - \big\|\sigma^{M_n^*}(t_{j,n}^*)\big\|_{\ell^2}\bigg)\,\Delta t_{j,n}^*\biggr| \\
    & \hspace{0.55cm}+ \biggr|\sum_{j = 0}^{k_n^* - 1} \big\|\sigma^{M_n^*}(t_{j,n}^*)\big\|_{\ell^2}\,\Delta t_{j,n}^* - \int_{0}^T \|\sigma(t)\|_{\ell^2}\,\text{d}t\biggr| \\
    & \leq \sum_{j=0}^{k_n^* - 1} \varepsilon_n \Delta t_{j,n}^* + K_1\left(\frac{T}{n\varepsilon_n} + \delta(M_n^*)\right) \\ 
    &\leq K_2 \Big(\varepsilon_n + (n\varepsilon_n)^{-1} + \delta(M_n^*)\Big).
    \end{split}
\end{equation}
Since from \eqref{eq:stepsize_s_star_def} and \eqref{eq:stepsize_s_dash_def}
\begin{equation*}
    \Big|\widetilde{S}_{M_n^*,n}^{\,1} - \hat{S}_{M_n^*,n}^{\,1}\Big| = \max\left\{\varepsilon_n, \big\|\sigma^{M_n^*}(\hat{t}_{k_{n}^*-1,n})\big\|_{\ell^2}\right\}\big(\hat{t}_{k_{n}^*,n} - T\big) \leq \frac{K_4}{n\varepsilon_n},
\end{equation*}
the inequality \eqref{eq:step_sums2} implies
\begin{equation}\label{eq:step_sums3}
    \biggr|\hat{S}_{M_n^*,n}^{\,1} -  \int_{0}^T \|\sigma(t)\|_{\ell^2}\,\text{d}t\,\biggr| \leq K_5\Big(\varepsilon_n + (n\varepsilon_n)^{-1} + \delta(M_n^*)\Big).
\end{equation}
Also, note that \eqref{eq:step_size} results in
\begin{equation}\label{eq:step_sum_hat}
    \hat{S}_{M_n^*,n}^{\,l} = \sum_{j = 0}^{k_n^* - 1}\left(\frac{T}{n}\right)^l = k_n^*\left(\frac{T}{n}\right)^l, \quad l \in \{1,2\}.
\end{equation}
Consequently, since $k_n^* = \mathcal{O}(n)$ by Fact \ref{fact:stepsize}, from \eqref{eq:step_sum_hat} it follows
\begin{equation}\label{eq:bound_s_sum}
\hat{S}_{M_n^*,n}^{\,l} \leq \tilde{C}, \quad n\in\mathbb{N}, \ l=1,2,
\end{equation}
for some $\tilde{C}>0.$
Consequently, \eqref{eq:step_sums3} together with \eqref{eq:step_sum_hat} yield
\begin{equation}\label{eq:step_sum1_integral}
    \biggr|\frac{k_n^*\,T}{n} - \int_{0}^T \|\sigma(t)\|_{\ell^2}\,\text{d}t\,\biggr| \leq K_5\Big(\varepsilon_n + (n\varepsilon_n)^{-1} + \delta(M_n^*)\Big).
\end{equation}
Furthermore, by \eqref{eq:step_sum_hat} and \eqref{eq:step_sum1_integral} we arrive at
\begin{equation}\label{eq:step_sums_3}
    \begin{split}
    \biggr|k_n^*\,\hat{S}_{M_n^*,n}^{\,2} - \left(\int_{0}^T \|\sigma(t)\|_{\ell^2}\,\text{d}t\right)^2 \biggr| &\leq\, \biggr|\hat{S}_{M_n^*,n}^{\,1} +  \int_{0}^T \|\sigma(t)\|_{\ell^2}\,\text{d}t\,\biggr|
    \times K_5\Big(\varepsilon_n + (n\varepsilon_n)^{-1} + \delta(M_n^*)\Big) \\
    & \leq K_6\Big(\varepsilon_n + (n\varepsilon_n)^{-1} + \delta(M_n^*)\Big),
    \end{split}
\end{equation}
where $K_6$ does not depend on $n$ on the virtue of \eqref{eq:bound_s_sum}.
Since $t_{k_n^*,n}^* = T,$ we also have
\begin{equation}\label{eq:step_snkn_diff}
\begin{split}
    \Big|k_n^* \hat{S}_{M_n^*,n}^{\,2} - k_n^* \widetilde{S}_{M_n^*,n}^{\,2}\Big| & = k_n^* \max\left\{\varepsilon_n^2, \big\|\sigma^{M_n^*}(\hat{t}_{k_n^*-1,n})\big\|_{\ell^2}^{2}\right\}\Big((\hat{t}_{k_n^*,n} - \hat{t}_{k_n^*-1,n})^2 - (T - t_{k_n^*-1,n}^*)^2\Big) \\
    & \leq 2k_n^* \max\left\{\varepsilon_n^2, \big\|\sigma^{M_n^*}(\hat{t}_{k_n^*-1,n})\big\|_{\ell^2}^{2}\right\}\big(\hat{t}_{k_n^*,n} - \hat{t}_{k_n^*-1,n}\big)^2.
    \end{split}
\end{equation}
By \eqref{eq:step_sums_3}, \eqref{eq:step_snkn_diff}, and Fact \ref{fact:stepsize} we obtain
\begin{equation}\label{eq:step_snkn_diff_2}
    \Big|k_n^* \hat{S}_{M_n^*,n}^{\,2} - k_n^* \widetilde{S}_{M_n^*,n}^{\,2}\Big| \leq K_7 \frac{k_n^*}{(n\varepsilon_n)^2} = K_7 \frac{k_n^*}{n}\,\frac{1}{n\varepsilon_n^2}\leq K_8 (n\varepsilon_n^2)^{-1}.
\end{equation}
Combining \eqref{eq:step_sums_3} and \eqref{eq:step_snkn_diff_2} results in
\begin{equation}\label{eq:step_limit_square}
\begin{split}
    \biggr|k_n^*\,\widetilde{S}_{M_n^*,n}^{\,2} - \left(\int_{0}^T \|\sigma(t)\|_{\ell^2}\,\text{d}t\right)^2\biggr| & \leq K_6\Big(\varepsilon_n + (n\varepsilon_n)^{-1} + \delta(M_n^*)\Big) + K_8 (n\varepsilon_n^2)^{-1} \\
    & \leq K_9\Big(\varepsilon_n + (n\varepsilon_n^2)^{-1} + \delta(M_n^*)\Big).
    \end{split}
\end{equation}
On the other hand, by Proposition \ref{upper_bound}, Fact \ref{fact:stepsize}, and \eqref{eq:stepsize_vs_conditional} we have
\begin{equation}\label{eq:step_limit_change}
\begin{split}
    (k_n^*)^{1/2}\bigg|\big\|X - X_{M_n^*, k_n^*}^* \big\|_2 - \big\|\widetilde{X}_{M_n^*, k_n^*}^E &- \mathbb{E}\big(\widetilde{X}_{M_n^*, k_n^*}^E \big| \mathcal{N}_{M_n^*, k_n^*}^*(W)\big)\big\|_2 \bigg| \phantom{expression}\\
    & \leq (k_n^*)^{1/2}\,\big\|X -  \widetilde{X}_{M_n^*, k_n^*}^E\big\|_2 \\
    & \leq K_{10}\Big((n\varepsilon_n^2)^{-1} + (k_n^*)^{1/2}\delta(M_n^*)\Big).
    \end{split}
\end{equation}
for some $K_{10}>0$ which depends only on the constants $C_1, C_2, T.$
 
Therefore, from \eqref{eq:G_delta}, \eqref{eq:varepsilon_prop}, and \eqref{eq:step_limit_change} it follows
\begin{equation}\label{eq:step_limit_switch}
    \lim\limits_{n\to+\infty} (k_n^*)^{1/2}\,\big\|X - X_{M_n^*, k_n^*}^* \big\|_2 =  \lim\limits_{n\to+\infty}(k_n^*)^{1/2}\,\big\|\widetilde{X}_{M_n^*, k_n^*}^E - \mathbb{E}\big(\widetilde{X}_{M_n^*, k_n^*}^E \big| \mathcal{N}_{M_n^*, k_n^*}^*(W)\big)\big\|_2.
\end{equation}
Rewriting \eqref{eq:step_limit_switch} analogously as in \eqref{eq:euler_minus_cond} and \eqref{eq:conditional_error}, we conclude that \eqref{eq:step_limit_square} and \eqref{eq:step_limit_switch} imply
\begin{equation}\label{eq:step_cnoneq_limit_1}
    \lim\limits_{n\to+\infty} (k_n^*)^{1/2}\,\big\|X - X_{M_n^*, k_n^*}^* \big\|_2 =  \lim\limits_{n\to+\infty} \biggr(\frac{k_n^*}{6}\,\tilde{S}^2_{M_n^*,n}\biggr)^{1/2} = \mathcal{C}_{noneq}.
\end{equation}
Hence, \eqref{eq:step_cnoneq_limit_1} yields
\begin{equation*}
    \big(\cost(X^{step}_{M_n^*,k_n^*})\big)^{1/2}\,\big\|X - X^{*}_{M_n^*,k_n^*}\big\|_2 \approx (M_n^*)^{1/2}\,\mathcal{C}_{noneq}, \quad n\to +\infty,
\end{equation*}
which establishes lower bound in \eqref{th:AC_noneq} for class $\chi_{noneq}^{\bar{M}}.$

\subsection{Asymptotically (almost) optimal algorithms for classes  \texorpdfstring{$\chi_{eq}, \chi_{noneq}$.}{}}\label{sec:optimal_alg}
\phantom{a}\vspace{0.17cm}\newline
In this subsection we extend optimality results obtained for fixed truncation level sequences $\bar{M}$ to the general classes of considered methods $\chi_{eq}, \chi_{noneq}$. Our final conclusions are gathered in Theorem \ref{th:AC_final}.

\begin{theorem}\label{th:AC_final}
Let $a,\sigma$ satisfy conditions (A1)-(A4) and (S1)-(S3) with sequence $\delta$, respectively. Let also $\diamond \in \{noneq, eq\}.$ Then, for every method $\bar{X}=(\xalg)_{n=1}^{\infty} \in \chi_{\diamond}$ we have
\vspace{0.05cm}
\begin{equation}\label{th:final_AC}
        \big(\cost(\xalg)\big)^{1/2}\,\big\| X - \xalg\big\|_2  \,\gtrapprox\,\big(\delta^{-1}(n^{-1/2})\big)^{1/2}\,\mathcal{C}_{\diamond}, \quad n\to +\infty.
\end{equation}

\vspace{0.09cm}
\noindent Moreover, for every truncation level sequence $M_n$ with $\delta^{-1}(n^{-1/2})=o(M_n), \ n\to +\infty,$  there exists a sequence $M^* = (M^*_n)_{n=1}^{\infty}\in\mathcal{G}_{\delta}$ such that $M_n^* = o (M_n), \ n\to+\infty,$ and:
\smallskip

a) the truncated-dimension Euler algorithm with adaptive path-independent step-size $X^{*} = \big(X^*_{M_n^*, k_n^*}\big)_{n=1}^{\infty}\in \chi_{noneq}$ satisfying
\begin{equation}\label{th:final_AC_noneq}
        \big(\cost(X_{M_n^*,k_n^*}^*)\big)^{1/2}\big\| X - X_{M_n^*,k_n^*}^*\big\|_2  \,\lessapprox\,\sqrt{\frac{M_n^*}{6}}\,\int_{0}^T \|\sigma(t)\|_{\ell^2}\,\rd t, \quad n\to +\infty;
\end{equation}
\smallskip

b) the truncated-dimension Euler algorithm $X^{Eq*} = \big(X^{Eq*}_{M_n^*, n}\big)_{n=1}^{\infty}\in \chi_{eq},$ based on the sequence of equidistant meshes, and satisfying
\begin{equation}\label{th:final_AC_eq}
        \big(\cost(X^{Eq*}_{M_n^*, n})\big)^{1/2}\big\| X - X^{Eq*}_{M_n^*, n}\big\|_2  \,\lessapprox\,\sqrt{\frac{M_n^* T}{6}}\,\left(\int_{0}^T \|\sigma(t)\|_{\ell^2}^2\,\rd t\right)^{1/2}, \quad n\to +\infty.
\end{equation}
\end{theorem}

\medskip

\begin{proof}
First, we note that \eqref{eq:G_delta} and monotonicity of $\delta^{-1}$ imply that for every admissible truncation level sequence $\bar{M} = (M_n)_{n=1}^{\infty}\in \mathcal{G}_{\delta}$ it holds $M_n \gtrapprox \delta^{-1}\big(n^{-1/2}\big)$. This in turn implies that \eqref{th:final_AC} is satified for every $\bar{M}$ and method $\overline{X}_{M_n, n}$ in $\chi_{noneq}$ and $\chi_{eq},$ respectively. 

However, the lower bound in \eqref{th:final_AC} cannot be asymptotically attained by any algorithm. Otherwise, the truncation level would satisfy $M_n \approx \delta^{-1}(n^{-1/2}),$ which in turn would violate the property \eqref{eq:G_delta}. 
Nevertheless, for every sequence $\bar{M}:=(M_n)_{n=1}^{\infty}\in\mathcal{G}_\delta$ there exists a constant $k_0^{\bar{M}} = k_0(\bar{M}) \in \mathbb{N}$ such that for every $k\geq k_0^{\bar{M}}$ we have
\begin{equation}\label{eq:gamma_seq}
    M_n \gtrapprox \,M_n\big(\log(\ldots\log(\log(n))\ldots )\big)^{-1}, \quad n \to +\infty,
\end{equation}
and both sequences in \eqref{eq:gamma_seq} belong to $\mathcal{G}_\delta,$ with natural logarithm being composed $k$-times. For fixed $k,$ denote the sequence on the right side of \eqref{eq:gamma_seq} by  $\bar{M}^k=(\bar{M}^k_n)_{n=1}^{\infty},$ and the output of an optimal method $(X^{*,k}_{\diamond}) = \big(X^{*,k}_{\diamond}(t)\big)_{t\in[0,T]}$ in corresponding class $\chi_{\diamond}^{\bar{M}^{k}}, \ \diamond \in \{eq, noneq\},$ respectively. For every $k\in \mathbb{N},$ $X^{*,k}_{eq}$ is a truncated dimension Euler scheme based on equidistant mesh, while $X^{*,k}_{noneq}$ is a truncated dimension Euler scheme with adaptive path-independent step-size control \eqref{eq:step_size}.

As a result, we have the following relation between the exact asymptotic behaviour of the method errors
\begin{equation*}
\begin{split}
    \big(\cost(\xalg)\big)^{1/2}\,\big\| X - \xalg\big\|_2  &\gtrapprox \,\big(\bar{M}^k_n\big)^{1/2}\,\mathcal{C}_{\diamond}\\
    & \approx \big(\cost(X^{*,k}_{\diamond})\big)^{1/2}\,\big\| X - X^{*,k}_{\diamond}\big\|_2, \quad k\geq k_0^{\bar{M}}, \quad n\to+\infty,
    \end{split}
\end{equation*}
holding for all $\xalg\in \chi_{\diamond}^{\bar{M}}.$ Since the constructed sequence satisfies $\bar{M}^k_n = o(M_n), \ n\to +\infty, $ this concludes the proof.
\end{proof}
\smallskip\vspace{-0.1cm}

\begin{remark}
Due to leveraging the alternative truncation sequence as per \eqref{eq:gamma_seq}, the asymptotic benefit is at least of order $\left(M_n / \bar{M}^k_n\right)^{1/2},$ which is unbounded and diverges to infinity, as $n\to+\infty$. In particular, this ratio can be significant also for smaller values of $n,$ especially for sequences $\delta$ converging relatively slowly to zero.
\end{remark}

\begin{remark}
In class $\chi_{noneq}$ we consider only algorithms which are non-adaptive with respect to the number of leveraged Wiener process coordinates. In particular, this results in lower bound for cost times error terms in \eqref{th:AC_noneq}, \eqref{th:AC_eq} diverge to infinity, as the cost rises. On the other hand, this is not the case for finite-dimensional noise structure, see e.g., \cite{AKPP}, \cite{PP2}. We conjecture that this term can be significantly lowered when additional adaptation with respect to Wiener process coordinates is introduced. We plan to investigate such algorithms in our future work. 
\end{remark}

\section{Numerical experiments and implementation issues in Python}\label{sec:experiments}

\subsection{Solver implementation in Numba}
\phantom{a}\vspace{0.17cm}\newline
In this section we exhibit alternative implementation of $\xstep,$ which leverages Numba compiler in Python. This enables us to execute specified functions, called kernels, on the GPU (Graphics Processing Unit) device which supports CUDA API developed by NVIDIA. For more details on parallel computing and kernel execution for a specified grid of blocks and threads, we refer to \cite{AKAPhD, PPMSLS}. Due to the fact that truncated dimension Euler scheme can be executed independently on each thread, we would expect significant decrease of the computation time when compared to the similar calculations on CPU (Central Processing Unit). This is beneficial especially when large number of the trajectories should be simulated. 

The crucial part of the code responsible for the algorithm $\xstep$ execution on GPU, together with relevant comments, can be found in the listing below.
\smallskip

\begin{lstlisting}[language=Python, caption=Truncated dimension Euler algorithm with adaptive step-size - crucial part of the code using Numba compiler. , basicstyle=\ttfamily\tiny]
# this decorator allows to execute the kernel on GPU device
@cuda.jit
def method_step_size(STATES, n, M_n, x0, timegrid, K, w_incr, out=[[]]):
    '''
    Simulate NUM_THREADS different trajectories by method X^{step}_{M_n,k_n} for fixed n and M_n.
    
    Parameters:
    STATES:                          initialised random number generator states 
                                     (separate for each single thread)
    n, M_n:                          algorithm parameters (n is not explicitely used here)
    timegrid:                        vector storing utilised scheme mesh 
                                     (based on step-size and derived outside of this function)
    K:                               integer; number of trajectories to be simulated
    out:                             (NUM_THREADS x k_n) - shaped array on GPU where 
                                     the simulated trajectories are stored
    '''
    
    # get thread info and the total number of threads in the one-dimensional grid, consisting of BLOCKS number of blocks and THREADS_PER_BLOCK # number of threads per each block
    start_thread_id = cuda.grid(1)
    stride = cuda.gridsize(1) 

    # each thread simulates trajectories with indices having specified remainder when divided by stride (i% stride)
    for i in range(start_thread_id, K+1, stride):
        
        # initialize the time moment for particular thread
        t = 0.0
        time_index = 0
        
        # initialize starting point
        x_step = x0

        # out[start_thread_id][:] stores the latest trajectory simulated by thread with index 'start_thread_id' within the grid. While keeping all K trajectories is also possible, this might cause memory issues for large K
        out[start_thread_id][0] = x0

        # initialize pivot variable storing intermediate scalar Wiener process' increment
        dW_1 = 0.0
        
        while t < T: 
            h = timegrid[time_index + 1] - timegrid[time_index]

            # initialize pivot variables needed for the diffusion-related term calculation
            single_sig_increment = 0.0
            while wiener_coord_index = 0
                
            while wiener_coord_index < int(M_n):
                # generate scalar float64 increments of the M_n - dimensional Wiener process 
                # we use xoroshiro128p_normal random number generator
                dW_1 = xoroshiro128p_normal_float64(STATES, start_thread_id) * sqrt(h) 
                single_sig_increment += sigma(t, wiener_coord_index + 1) * dW_1
                wiener_coord_index += 1

            # calculate single step of X^{step}_{M_n,k_n}
            x_step = x_step + a(t, x_step) * h + single_sig_increment

            # update entries of the vectors storing the values of the approximate trajectories
            out[start_thread_id][time_index + 1] = x_step

            # update time point
            time_index +=1
            t = timegrid[time_index]

\end{lstlisting}
\subsection{Results of numerical experiments in Python}
\phantom{a}\vspace{0.17cm}\newline
In this section, we present results of numerical experiments performed on CPU by using Multiprocessing library in Python programming language. We analyse the asymptotic error behaviour for both algorithms $\xstep$ and $X_{M_n^*,n}^{Eq}$ by verifying if the ratio between constants $\mathcal{C}_{noneq}$ and $\mathcal{C}_{eq},$ appearing in Theorem \ref{th:AC_lower}, is attained.

To this end, we consider the following parameters: $T =1.5, \ x_0 = 0.9,$ and the equation coefficients
\begin{equation}\label{eq:sim_a}
    a(t,x) = (t+2)(x-1), \quad t\in [0,T], \ x\in\mathbb{R},
\end{equation}

\begin{equation}\label{eq:sim_sig}
    \sigma_k(t) = \frac{e^{2t} + 2}{(k+1)^{p}\sqrt{\log(k+1)}}\quad t\in [0,T], \ k = 1,2,\ldots,
\end{equation}
where $p>1/2.$
One can show that for all $l\in \mathbb{N}, \ l > 1$ it holds
\begin{equation*}
    \|\sigma(t) - P_l(t)\|_{\ell^2}^2 = \left|\sum_{k=l}^{+\infty} \frac{(e^{2t}+ 2)^2}{(k+1)^{2p}\log(k+1)}\right|\leq (e^{2T} + 2)^2\left|\int\limits_{l}^{+\infty}\frac{1}{(x+1)^{2p}\log(x+1)}\rd x\right|.
\end{equation*}
By substituting $v = (2p-1)\log(x+1),$ we arrive at
\begin{equation}\label{eq:gamma_integral}
    \|\sigma(t) - P_l(t)\|_{\ell^2}^2 \leq (e^{2T} + 2)^2\left|\,\int\limits_{(2p-1)\log(l+1)}^{+\infty}e^{-v}\,v^{-1}\rd v\right|,
\end{equation}
and the integral appearing in \eqref{eq:gamma_integral} is equal to the upper incomplete gamma function \linebreak $\Gamma\big(1, (2p-1)\log(l+1)\big).$ Since for all $s\in \mathbb{N}, x>0$ we have
\begin{equation}\label{eq:incomplete_gamma}
    \Gamma(s,x) = (s-1)!\,\,e^{-x}\sum_{k=0}^{s-1}\frac{x^k}{k!},
\end{equation}
combining \eqref{eq:gamma_integral} and \eqref{eq:incomplete_gamma} yields
\begin{equation*}
    \|\sigma(t) - P_l(t)\|_{\ell^2} \leq (e^{2T} + 2)e^{-0.5(2p-1)\log(l+1)} = (e^{2T} + 2)\big(l+1\big)^{1/2-p}.
\end{equation*}
Therefore, we can assume $\delta(n) \approx n^{1/2-p},$ which implies $\delta^{-1}(n^{-1/2})\approx n^{\frac{1}{2p-1}}.$ In our simulations, we set $p = 0.9,$ hence $M_n \gtrapprox n^{5/4 + \varepsilon}\in \mathcal{G}_\delta$ for all $\varepsilon > 0.$ We decide to choose $\varepsilon = 0.03.$ Consequently, it suffices to take $M_n = \Theta(n^{1.28}),$ $n \to +\infty$.

Now we provide the values of the constants $\mathcal{C}_{noneq}$ and $\mathcal{C}_{eq}$ in our model. First, for all $t\in [0,T]$ we have that
$$\|\sigma(t)\|_{\ell^2} = \gamma(e^{2t} + 2)\quad
\textrm{with}\quad \gamma = 0.75638883...$$
Finally, the constant appearing in lower bounds for $\chi_{eq}$ is equal to
\begin{equation*}
    \sqrt{\frac{T}{6}}\big(0.25e^{4T} + 2e^{2T} - 9/4 + 4T\big)^{1/2}\gamma = 4.55058060...,
\end{equation*}
while in $\chi_{noneq}$
\begin{equation*}
    \sqrt{\frac{1}{6}}\big(0.5e^{2T} - 0.5 + 2T\big)\gamma = 3.87313729... .
\end{equation*}

Since the analytical form of the unique solution to the equation \eqref{main_equation} with coefficients as per \eqref{eq:sim_a} and \eqref{eq:sim_sig} is not known, for each $X^{alg}\in \{X_{M_n^*,k_n^*}^{Eq*}, \xstep\}$ we execute in parallel the algorithm $X_{W_{ratio}\cdot M_n^*,n^*}^{Eq*}$ based on equidistant mesh with $n^* = 10^6$ nodes and first $W_{ratio}\cdot M_n^*$ coordinates of the countably dimensional Wiener process $W$. Let us denote the corresponding process by $X_{M_n^*,n^*}.$ The method error, $\textrm{err}_K(X^{alg}),$ is estimated by simulating $K$ trajectories of the underlying processes. We measure the difference between each pair of trajectories by using the composite Simpson quadrature $Q$ based on time points for which $X^{alg}$ is evaluated, together with the midpoints of the corresponding subintervals. To summarise, we take
\begin{equation}\label{eq:simpson}
\textrm{err}_K(X^{alg}) := \biggr(\frac{1}{K}\sum_{l=1}^{K}Q\Big(|X^{alg}_l(a,b,W^{(l)}) - X_{W_{ratio} M_n^*,n^*,l}\,(a,b,W^{(l)})|^2\Big)\biggr)^{1/2},
\end{equation}
where $X^{alg}_l,$ $X_{W_{ratio}\cdot M_n^*,n^*,l},$ and $W^{(l)}$ are the $l$--th generated trajectories of the corresponding processes. Finally, we compare empirical improvement ratio $\textrm{err}_K(\xstep) / \textrm{err}_K(X_{M_n^*,k_n^*}^{Eq*})$ with the theoretical value $\mathcal{C}_{noneq} / \mathcal{C}_{eq} \simeq 0.85113723.$ The testing results are exhibited in Table \ref{tab:experiments}.
\begin{center}
\scriptsize{
\begin{table}[ht]
\begin{tabular}{|c|c|c|c|c|c|}
\hline
 $n$ & $k_n^*$ & $M_n^*$ & $K$ & $W_{ratio}$ & Improvement ratio \\
 \hline
\hline
   1000 & 7832 & 1037 & 1000 & 2.0 & 0.977977 \\
 \hline
   2000 & 15686 & 2520 & 1000 & 2.0 & 0.945800 \\
 \hline
   5000 & 39249 & 8142 & 250 & 1.5 & 0.915620 \\
 \hline
   10000 & 78520 & 19773 & 94 & 1.5 & 0.976337 \\
 \hline
\end{tabular}
\caption{Simulation results for $X_{M_n^*,n}^{Eq*}$ and $\xstep.$}
\label{tab:experiments}
\end{table}
}
\end{center}
\FloatBarrier
\vspace{-1.6cm}
The average improvement from leveraging adaptive mesh is generally visible. For $n\in\{5000, \linebreak 10000\}$ we executed smaller number of trajectories due to high complexity and time consumption. We also note that the impact of leveraging \eqref{eq:simpson}, Monte Carlo simulation, and rare-fine mesh comparison as an approximation of the method error is not quantified. Nevertheless, the obtained ratios are roughly aligned with the expected asymptotic error behaviour.
\section{Conclusions}
We investigated global approximation of SDEs driven by countably dimensional Wiener process, where the diffusion term depends only on the time variable. For fixed sequence $\delta,$ modelling level of decay for the diffusion term, we derived lower bounds for asymptotic error in suitable classes of algorithms leveraging specified truncation levels of the Wiener process. In particular, we quantified asymptotic benefit from leveraging step-size control instead of equidistant mesh. We also constructed two truncated dimension Euler schemes which are the (almost) optimal algorithms in the respective classes. Our results indicate that the decrease of method error requires significant increase of the cost term, which is illustrated by the product of cost and minimal error diverging to infinity. Nonetheless, we conjecture that the estimates might be beaten in case we allow for additional adaptation with respect to different Wiener process coordinates. 
\section{Appendix}
\begin{fact}\label{fact:integral}
    Let $\sigma$ satisfy (S1)-(S3) with a sequence $\delta$. Then, for every $\bar{M} = (M_n)_{n=1}^\infty \in \mathcal{G}_\delta$ and an arbitrary sequence of partitions $(\Delta_n)_{n=1}^{\infty}$ with
    \begin{equation}\label{arbitary_partition_fact}
        \Delta_{n}: \quad 0=t_{0,n} < t_{1,n} < \ldots < t_{\bar{k}_n-1,n} < t_{\bar{k}_n,n} = T
    \end{equation}
    with $\bar{k}_n \in \mathbb{N},$ we have
    \begin{equation}\label{eq:conv_to_int}
        \biggr|\sum_{j=0}^{\bar{k}_n -1}\|\sigma^{M_n}(t_{j,n})\|_{\ell^2} \Delta t_{j,n} - \int_{0}^T \|\sigma(t)\|_{\ell^2}\rd t\biggr| \leq C_2 T\bigg(\max\limits_{0\leq j \leq \bar{k}_n-1}(\Delta t_{j,n}) + \delta(M_n)\bigg),
    \end{equation}
    where $\Delta t_{j,n} = t_{j+1,n} - t_{j,n}, \ n\in \mathbb{N}, \ j =0,\ldots, \bar{k}_n-1.$
\end{fact}

\begin{proof}[{\bf Proof of Fact 2.}]
Let us fix $n\in\mathbb{N}.$ Since 
\begin{equation}\label{fact:integral_proof_1}
\begin{split}
    \biggr|\sum_{j=0}^{\bar{k}_n - 1}\|\sigma^{M_n}(t_{j,n})\|_{\ell^2}\,\Delta t_{j,n} - \int_{0}^T \|\sigma(t)\|_{\ell^2}\rd t\biggr| &\leq \biggr|\sum_{j=0}^{\bar{k}_n - 1}\bigg(\|\sigma(t_{j,n})\|_{\ell^2} - \|\sigma^{M_n}(t_{j,n})\|_{\ell^2}\bigg) \Delta t_{j,n}\biggr| \\ 
    & \ + \biggr|\sum_{j=0}^{\bar{k}_n -1}\|\sigma(t_{j,n})\|_{\ell^2} \Delta t_{j,n} - \int_{0}^T \|\sigma(t)\|_{\ell^2}\rd t\biggr|,
\end{split}
\end{equation}
we split the proof into two parts, by estimating each term in 
\eqref{fact:integral_proof_1} separately.

First, by the property (S2), we obtain
\begin{equation}\label{fact:integral_proof_2}
\biggr|\sum_{j=0}^{\bar{k}_n - 1}\big(\|\sigma(t_{j,n})\|_{\ell^2} - \|\sigma^{M_n}(t_{j,n})\|_{\ell^2}\big) \Delta t_{j,n}\biggr| \leq C_2T\delta(M_n).
\end{equation}
Second, by applying $n$ times the mean value theorem for integrals we have
\begin{equation}\label{fact:integral_proof_3}
        \int_{0}^{T}\|\sigma(t)\|_{\ell^2}\rd t = \sum_{j=0}^{\bar{k}_n -1}\|\sigma(\xi_{j,n})\|_{\ell^2}\,\Delta t_{j,n}, \quad \xi_{j,n} \in (t_{j,n}, t_{j+1,n}).
\end{equation}
Therefore, by (S2) and \eqref{fact:integral_proof_3}, it holds
\begin{equation}\label{fact:integral_proof_4}
    \begin{split}
        \biggr|\sum_{j=0}^{\bar{k}_n-1}\|\sigma(t_{j,n})\|_{\ell^2} \Delta t_{j,n} - \int_{0}^T \|\sigma(t)\|_{\ell^2}\rd t\biggr|
        & \leq \sum_{j=0}^{\bar{k}_n -1}\Big|\|\sigma(\xi_{j,n})\|_{\ell^2} - \|\sigma(t_{j,n})\|_{\ell^2}\Big|\Delta t_{j,n} \\
        & \leq C_2 T \max\limits_{0 \leq j \leq \bar{k}_n-1}\,\Delta t_{j,n}.
    \end{split}
\end{equation}
Now, the assertion of the fact follows from \eqref{fact:integral_proof_1}, \eqref{fact:integral_proof_2}, and \eqref{fact:integral_proof_4}.
\end{proof}

\noindent\textbf{Proof of Proposition \ref{prop:findim}.}
For fixed $M,n \in \mathbb{N},$ we use the following decomposition of the process $X$ and the truncated dimension time-continuous Euler process
 \begin{equation*}
     X^M(t) = x_0 + A_M(t) + B^M(t),
 \end{equation*}
 \begin{equation*}
     \euler(t) = x_0 + \tilde{A}_{M,n}(t) + \tilde{B}_{n}^M(t),
 \end{equation*}
 where
 \begin{equation*}
     A_{M}(t) = \int\limits_{0}^t \sum_{j=0}^{k_n-1} a(s, X^M(s))\one_{(t_{j,n}, t_{j+1,n}]}(s)\rd s,
 \end{equation*}
 \begin{equation*}
     B^{M}(t) = \int\limits_{0}^t \sum_{j=0}^{k_n-1} \sigma^M(s)\one_{(t_{j,n}, t_{j+1,n}]}(s)\rd W(s), 
 \end{equation*}
 and
 \begin{equation*}
     \tilde{A}_{M,n}(t) = \int\limits_{0}^t \sum_{j=0}^{k_n-1} a(t_{j,n}, \euler(t_{j,n}))\one_{(t_{j,n}, t_{j+1,n}]}(s)\rd s,
 \end{equation*}
 \begin{equation*}
     \tilde{B}_{n}^M(t) = \int\limits_{0}^t \sum_{j=0}^{k_n-1} \sigma^M(t_{j,n})\one_{(t_{j,n}, t_{j+1,n}]}(s)\rd W(s).
 \end{equation*}
 Therefore,
 \begin{equation}\label{eq:euler_a}
     \mathbb{E}|A_M(t) - \tilde{A}_{M,n}(t)|^2 \leq 3\Bigl(\mathbb{E}\bigl|\tilde{A}_{n,1}^{E, M}(t)\bigr|^2 + \mathbb{E}\bigl|\tilde{A}_{n,2}^{E, M}(t)\bigr|^2 + \mathbb{E}\bigl|\tilde{A}_{n,3}^{E, M}(t)\bigr|^2\Bigr)
 \end{equation}
 with
 \begin{equation*}
     \mathbb{E}\bigl|\tilde{A}_{n,1}^{E, M}(t)\bigr|^2 = \mathbb{E}\biggr|\int\limits_{0}^t \sum_{j=0}^{k_n-1}\big(a(s, X^M(s) - a(t_{j,n},X^M(s))\big)\one_{(t_{j,n}, t_{j+1,n}]}(s)\rd s\biggr|^2,
 \end{equation*}
 \begin{equation*}
     \mathbb{E}\bigl|\tilde{A}_{n,2}^{E, M}(t)\bigr|^2 = \mathbb{E}\biggr|\int\limits_{0}^t \sum_{j=0}^{k_n-1}\big(a(t_{j,n}, X^M(s) - a(t_{j,n},X^M(t_{j,n}))\big)\one_{(t_{j,n}, t_{j+1,n}]}(s)\rd s\biggr|^2,
 \end{equation*}
 \begin{equation*}
     \mathbb{E}\bigl|\tilde{A}_{n,3}^{E, M}(t)\bigr|^2 = \mathbb{E}\biggr|\int\limits_{0}^t \sum_{j=0}^{k_n-1}\big(a(t_{j,n},X^M(t_{j,n}) - a(t_{j,n},\euler(t_j))\big)\one_{(t_{j,n}, t_{j+1,n}]}(s)\rd s\biggr|^2.
 \end{equation*}
 From the H\"{o}lder inequality and Lemma \ref{lemma_scheme_bound} we have that
 \begin{equation}\label{eq:euler_a_1}
     \mathbb{E}\bigl|\tilde{A}_{n,1}^{E, M}(t)\bigr|^2 \leq T \sum_{j=0}^{k_n-1}\int_{t_{j,n}}^{t_{j+1,n}}|s-t_{j,n}|^2\,\mathbb{E}\big(1+ |X^M(s)|\big)^2 \rd s \leq C \max\limits_{0\leq j \leq k_n-1}{(\Delta t_{j,n})^2}.
 \end{equation}
By Corollary 14.2.9 in \cite{CohEl} we get that
\begin{equation*}
\begin{split}
     a\big(t_{j,n}, X^M(s)\big) - a\big(t_{j,n}, X^M(t_{j,n})\big)& =\int_{t_{j,n}}^s \bigg[\frac{\partial  a}{\partial x}\big(t_{j,n}, X^M(u)\big)\cdot a\big(u, X^M(u)\big) \\
     & \quad + \frac{1}{2}\frac{\partial^2 a}{\partial x^2}\big(t_{j,n}, X^M(u)\big)\cdot\|\sigma^M(u)\|^2_{\ell^2}\bigg]\rd u \\ 
     & \quad + \int_{t_{j,n}}^s \frac{\partial  a}{\partial x}\big(t_{j,n}, X^M(u)\big)\sigma^M(u)\rd W(u).
\end{split}
\end{equation*}
For brevity, we introduce the following notation for $u \in (t_{j,n},t_{j+1,n}]$
\begin{equation*}
\begin{split}
     \alpha_{M,j}(u) = \frac{\partial  a}{\partial x}\big(t_{j,n}, X^M(u))\cdot a\big(u, X^M(u)\big) + \frac{1}{2}\frac{\partial^2 a}{\partial x^2}\big(t_{j,n}, X^M(u)\big)\cdot\|\sigma^M(u)\|^2_{\ell^2}\bigg]\rd u \in \mathbb{R},
\end{split}
\end{equation*}

\begin{equation}\label{eq:euler_a_2}
\begin{split}
    \beta_j^M(u) &= \frac{\partial  a}{\partial x}(t_{j,n}, X^M(u))\cdot \sigma^M(u) \\
    & = \bigg[\frac{\partial  a}{\partial x}\big(t_{j,n}, X^M(u)\big)\cdot \sigma_1(u), \ldots, \frac{\partial  a}{\partial x}\big(t_{j,n}, X^M(u)\big)\cdot \sigma_M(u), 0,0, \ldots\bigg] \in \mathbb{R}^{1 \times \infty},
    \end{split}
\end{equation}
and
\begin{equation*}
    \mathds{1}_{j,n}(u) = \mathds{1}_{(t_{j,n}, t_{j+1,n}]}(u), \quad u \in [0,T], \ j=0,\ldots, k_n-1.
\end{equation*}
Therefore, we have for all $t \in [0,T]$ that
\begin{equation}\label{eq:euler_2_split}
     \mathbb{E}\bigl|\tilde{A}_{n,2}^{E, M}(t)\bigr|^2 \leq 2\Big(\mathbb{E}\bigl|\tilde{A}_{n,21}^{E, M}(t)\bigr|^2 + \mathbb{E}\bigl|\tilde{A}_{n,22}^{E, M}(t)\bigr|^2 \Big),
 \end{equation}
where
\begin{equation*}
     \mathbb{E}\bigl|\tilde{A}_{n,21}^{E, M}(t)\bigr|^2 = \mathbb{E} \biggr|\int_0^t \,\sum_{j=0}^{k_n-1}\biggr(\int_{t_j}^s\alpha_{M,j}(u)\rd u\biggr)\cdot\mathds{1}_{j,n}(s)\rd s \biggr|^2,
\end{equation*}
\begin{equation*}
     \mathbb{E}\bigl|\tilde{A}_{n,22}^{E, M}(t)\bigr|^2 = \mathbb{E} \biggr|\int_0^t \,\sum_{j=0}^{k_n-1}\biggr(\int_{t_j}^s\beta_{j}^M(u)\rd W(u)\biggr)\cdot\mathds{1}_{j,n}(s)\rd s \biggr|^2.
\end{equation*}
Moreover, by the H\"{o}lder inequality, \eqref{eq:bounded_derivatives} and Lemma \ref{lemma_sol_bound} we obtain for $t \in [0,T]$ that
\begin{equation}\label{eq:euler_a_21_bound}
\begin{split}
     \hspace{-0.6cm}\mathbb{E}\bigl|\tilde{A}_{n,21}^{E, M}(t)\bigr|^2 & \\ &\hspace{-1.6cm}\leq C \sum_{j=0}^{k_n-1}\int_{t_{j,n}}^{t_{j+1,n}} \biggr[(s-t_{j,n})\int_{t_{j,n}}^s \Big(1 + \mathbb{E}|X^M(u)|^2 + \|\sigma^M(u) - \sigma^M(0)\|_{\ell^2}^4 + \|\sigma^M(0)\|_{\ell^2}^4\Big)\rd u\biggr] \rd s \\
     & \hspace{-1.6cm}\leq C\max\limits_{0\leq j \leq k_n-1}(\Delta t_{j,n})^2.
     \end{split}
\end{equation}
By the Fubini theorem and the fact that It\^{o} integrals defined on disjoint intervals are uncorrelated, for $t \in [0,T]$ we get that
\begin{equation*}
\begin{split}
     \hspace{-0.5cm}\mathbb{E}\bigl|\tilde{A}_{n,22}^{E, M}(t)\bigr|^2 & = \sum_{j=0}^{k_n-1}\int_{0}^{t}\int_{0}^{t}\mathbb{E}\bigg[\int_{t_{j,n}}^{s_1}\beta_j^M(u)\rd W(u) \cdot \int_{t_{j,n}}^{s_2}\beta_j^M(u)\rd W(u)\bigg]\cdot\mathds{1}_{j,n}(s_1)\,\mathds{1}_{j,n}(s_2)\rd s_1\rd s_2 \\
     &\hspace{-0.6cm}+ \sum_{j,l=0; l\neq j}^{k_n-1}\int_{0}^{t}\int_{0}^{t}\mathbb{E}\bigg[\int_{t_l}^{s_1}\beta_l^M(u)\rd W(u) \cdot \int_{t_j}^{s_2}\beta_j^M(u)\rd W(u)\bigg]\cdot\mathds{1}_{l,n}(s_1)\,\mathds{1}_{j,n}(s_2)\rd s_1\rd s_2 \\
     &\hspace{-0.6cm}= \sum_{j=0}^{k_n-1}\int_{t_{j,n}}^{t_{j+1,n}}\int_{t_{j,n}}^{t_{j+1,n}}\mathbb{E}\bigg[\int_{t_{j,n}}^{s_1}\beta_j^M(u)\rd W(u) \cdot \int_{t_{j,n}}^{s_2}\beta_j^M(u)\rd W(u)\bigg]\rd s_1\rd s_2.
     \end{split}
\end{equation*}
From Theorem 88 (iii) in \cite{situ}, \eqref{eq:bounded_derivatives}, and \eqref{eq:euler_a_2} we obtain
\begin{equation}\label{eq:euler_a_22}
    \begin{split}
     \mathbb{E}\bigl|\tilde{A}_{n,22}^{E, M}(t)\bigr|^2 &= \sum_{j=0}^{k_n-1}\,\int_{t_{j,n}}^{t_{j+1,n}}\int_{t_{j,n}}^{t_{j+1,n}}\mathbb{E}\bigg|\int_{t_{j,n}}^{s_1 \wedge s_2}\beta_j^M(u)\rd W(u) \bigg|^{\textcolor{red}{2}}\rd s_1\rd s_2 \\
     & = \sum_{j=0}^{k_n-1}\int_{t_{j,n}}^{t_{j+1,n}}\int_{t_{j,n}}^{t_{j+1,n}}\mathbb{E}\int_{t_{j,n}}^{s_1 \wedge s_2}\|\beta_j^M(u)\|^2_{\ell^2}\rd u \rd s_1\rd s_2 \\
     &\leq C\max\limits_{0\leq j \leq k_n-1}(\Delta t_{j,n})^2.
     \end{split}
\end{equation}
Combining \eqref{eq:euler_2_split}, \eqref{eq:euler_a_21_bound}, and \eqref{eq:euler_a_22} yields
\begin{equation}\label{eq:euler_a_2final}
    \mathbb{E}\bigl|\tilde{A}_{n,2}^{E, M}(t)\bigr|^2 \leq C\max\limits_{0\leq j \leq k_n-1}(\Delta t_{j,n})^2.
\end{equation}
In addition, we estimate the term
\begin{equation}\label{eq:euler_a_3}
     \mathbb{E}\bigl|\tilde{A}_{n,3}^{E, M}\bigr|^2 \leq T \int_0^t \sum_{j=0}^{k_n-1}\,\mathbb{E}|X^M(t_{j,n}) - \euler(t_{j,n})|^2\,\cdot\mathds{1}_{j,n}(s) \rd s.
 \end{equation}
Finally, by \eqref{eq:euler_a}, \eqref{eq:euler_a_1}, \eqref{eq:euler_a_2final}, and \eqref{eq:euler_a_3}, for all $t\in [0,T]$ it holds
\begin{equation}\label{eq:euler_a_final}
     \mathbb{E}|A_M(t) - \tilde{A}_{M,n}(t)|^2 \leq D_1\max\limits_{0\leq j \leq k_n-1}(\Delta t_{j,n})^2 + D_2 \int_0^t \sum_{j=0}^{k_n-1}\mathbb{E}|X^M(t_{j,n}) - \euler(t_{j,n})|^2\cdot\mathds{1}_{j,n}(s) \rd s, 
 \end{equation}
where the constants $D_1, D_2$ do not depend on the parameters $M,n.$

Now we estimate the diffusion-related term. By (S2) we have for $t\in [0,T]$ that
 \begin{equation}\label{eq:euler_b_final}
     \begin{split}
         \mathbb{E}|B^M(t) - \tilde{B}_{n}^M (t)|^2 & = \sum_{j=0}^{k_n-1}\mathbb{E}\int_0^t \|\sigma^M(s) - \sigma^M(t_{j,n})\|^2_{\ell^2}\cdot\mathds{1}_{j,n}(s)\rd s \\
         & \leq \frac{1}{3} C_2^2\,T\max\limits_{0\leq j \leq k_n-1}(\Delta t_{j,n})^2.
     \end{split}
 \end{equation}
By \eqref{eq:euler_a_final} and \eqref{eq:euler_b_final} we get that for all $t \in [0,T]$
 \begin{equation*}
     \sup_{0\leq s \leq t}\mathbb{E}|X^M(s) - \euler(s)|^2 \leq D_3 \max\limits_{0\leq j \leq k_n-1}(\Delta t_{j,n})^2 + D_4 \int_0^t \sup_{0\leq u \leq s}\mathbb{E}|X^M(u) - \euler(u)|^2\rd s.
 \end{equation*}
 Note that the mapping $\displaystyle{[0,T]\ni t \mapsto  \sup_{0 \leq u \leq t}\mathbb{E}\|X^M(u) - \euler(u)\|^p}$ is Borel (as a non-decreasing function) and bounded. Therefore, the Gr\"onwall's lemma yields
\begin{equation}\label{euler_final_estimate}
		\sup_{0 \leq t \leq T}\mathbb{E}\big|X^M(t) - \euler(t)\big|^2 \leq D_{0}\max\limits_{0\leq j \leq k_n-1}(\Delta t_{j,n})^2,
\end{equation}
where $D_{0}$ does not depend on truncation parameter $M$ and mesh size $k_n$. Now, \eqref{Euler_error_upper_bound} is a direct consequence of \eqref{euler_final_estimate}. \ \ $\square$

\begin{lemma}
\label{lemma_scheme_bound}
	Let $a,\sigma$ satisfy conditions (A1)-(A4) and (S1)-(S3), respectively. There exists $K \in (0,+\infty)$ such that for every $M,n\in\mathbb{N}$ it holds
	\begin{equation*}
		\sup_{0 \leq t \leq T}\mathbb{E}|\euler(t)|^2 \leq K.
	\end{equation*}
\end{lemma}
\begin{proof}
This property of the truncated dimension time-continuous Euler scheme has been shown for more generalised model structure in \cite{PPMSLS} for pointwise approximation problem. We refer to this paper for an outline of the proof.
\end{proof}

\end{document}